\documentclass{amsart}
\usepackage{amsmath}
\usepackage{amssymb}
\usepackage{amsfonts}

\newtheorem{theorem}{Theorem}
\theoremstyle{plain}
\newtheorem{acknowledgement}{Acknowledgement}

\newtheorem{corollary}{Corollary}

\newtheorem{definition}{Definition}

\newtheorem{lemma}{Lemma}

\numberwithin{equation}{section}
\input{tcilatex}

\begin{document}
\title[ rough generalized commutators{\ with Lipschitz functions}]{The
	boundedness of rough generalized commutators{\ with Lipschitz functions on }%
	homogeneous variable exponent Herz type spaces}
\author{FER\.{I}T G\"{U}RB\"{U}Z}
\address{Department of Mathematics, K\i rklareli University, K\i rklareli
39100, T\"{u}rkiye }
\email{feritgurbuz@klu.edu.tr}
\urladdr{}
\thanks{}
\curraddr{ }
\urladdr{}
\thanks{}
\date{}
\dedicatory{}
\thanks{}

\begin{abstract}
With the development of science, many nonlinear problems have emerged. At
this time, the classical function space has certain restrictions. For
example, it has lost its effectiveness for nonlinear problems under
nonstandard growth conditions. In the process of studying such nonlinear
problems, scholars are paying more and more attention to the transition from
classical function space to variable exponent function space. Also, there is
a big difference between variable exponent space and classical function
space, mainly because variable exponent function space has lost translation
invariance. This difference leads to many properties that hold in classical
space no longer hold in variable exponent space. It is important to
emphasize that variable exponent function spaces are a fundamental building
block in harmonic analysis. In recent years, there has been a growing
interest in the study of function spaces equipped with variable exponents,
leading to the development of a new framework known as variable exponent
analysis. These spaces provide a powerful tool for analyzing functions with
variable growth or decay rates and have found applications in various areas
of mathematics, including partial differential equations, harmonic analysis
and image processing. One can better understand the heterogeneity and
complexity inherent in many real world phenomena by taking into
consideration the theory of variable exponent function spaces. Thus, by
using certain properties of Lipschitz functions and variable exponents, in
this article, we establish the boundedness of a class of rough generalized
commutators {with Lipschitz functions} on{\ }homogeneous variable exponent
Herz and Herz-Morrey spaces.
\end{abstract}

\maketitle

\section{Introduction and Main Results}

Assume that $\Omega \in L^{s}(S^{n-1})\left( s>1\right) $, $\Omega $ is
homogeneous of degree zero on ${\mathbb{R}^{n}}$ with zero mean value on $%
S^{n-1}$, $S^{n-1}$ denotes the unit sphere on ${\mathbb{R}^{n}}$, $m$ is a
positive integer, $A\left( x\right) $ is a function defined on ${\mathbb{R}%
	^{n}}$ with $m$-th order derivatives on $L_{loc}\left( {\mathbb{R}^{n}}%
\right) $ and%
\begin{equation*}
	L_{loc}(%
	\mathbb{R}
	^{n})=\left \{ f:\int \limits_{K}\left \vert f\left( x\right) \right \vert
	dx<\infty ;\text{ for all compact subset }K\subset 
	\mathbb{R}
	^{n}\right \} .
\end{equation*}%
In analysis, theory and applications, the generalized commutators are
popular operators extensively studied over the past hundred years and
subjected to many generalizations in various settings. There are many
others, but we will limit ourselves to these two, for these are the main
focus of our objective. Thus, in this paper, we investigate the following
generalized commutators

\begin{equation*}
	I_{\Omega ,\phi }^{A,m}f(x)=\int \limits_{{\mathbb{R}^{n}}}\frac{\Omega (x-y)%
	}{|x-y|^{n-\phi +m-1}}R_{m}\left( A;x,y\right) f(y)dy\qquad 0<\phi <n
\end{equation*}%
and%
\begin{equation*}
	M_{\Omega ,\phi }^{A,m}f(x)=\sup_{r>0}\frac{1}{r^{n-\phi +m-1}}\int
	\limits_{|x-y|<r}\left \vert \Omega (x-y)R_{m}\left( A;x,y\right) f(y)\right
	\vert dy\qquad 0<\phi <n,
\end{equation*}%
where $m\in 
\mathbb{N}
$, and $R_{m}\left( A;x,y\right) $ denotes the $m$-th remainder of Taylor
series of $A$ at $x$ about $y$, more precisely, 
\begin{equation*}
	R_{m}\left( A;x,y\right) =A\left( x\right) -\sum \limits_{\left \vert
		\gamma \right \vert <m}\frac{1}{\gamma !}D^{\gamma }A\left( y\right) \left(
	x-y\right) ^{\gamma },
\end{equation*}%
where $D^{\gamma }A\in L^{r}\left( 
\mathbb{R}
^{n}\right) $ $\left( 1<r\leq \infty \right) $, $D^{\gamma }A\in BMO\left( 
\mathbb{R}
^{n}\right) $ or $D^{\gamma }A\in \dot{\Lambda}_{\beta }\left( {\mathbb{R}}%
^{n}\right) $ for $\left \vert \gamma \right \vert =m-1$ $\left( m\geq
1\right) $.

When $m=1$ above, it is obvious to see that $R_{1}\left( A;x,y\right)
=A\left( x\right) -A\left( y\right) $. In this case, $I_{\Omega ,\phi
}^{A,1}=I_{\Omega ,\phi }^{A}$ and $M_{\Omega ,\phi }^{A,1}=M_{\Omega ,\phi
}^{A}$ are just commutator operators, 
\begin{eqnarray*}
	I_{\Omega ,\phi }^{A}f\left( x\right) &=&A\left( x\right) I_{\Omega ,\phi
	}f\left( x\right) -I_{\Omega ,\phi }\left( Af\right) \left( x\right) \\
	&=&\int \limits_{{\mathbb{R}^{n}}}\frac{\Omega (x-y)}{|x-y|^{n-\phi }}\left(
	A\left( x\right) -A\left( y\right) \right) f(y)dy\qquad 0<\phi <n
\end{eqnarray*}%
and%
\begin{eqnarray*}
	M_{\Omega ,\phi }^{A}f\left( x\right) &=&A\left( x\right) M_{\Omega ,\phi
	}f\left( x\right) -M_{\Omega ,\phi }\left( Af\right) \left( x\right) \\
	&=&\sup_{r>0}\frac{1}{r^{n-\phi }}\int \limits_{|x-y|<r}\left \vert \Omega
	\left( x-y\right) \right \vert \left \vert A\left( x\right) -A\left(
	y\right) \right \vert \left \vert f(y)\right \vert dy\qquad 0<\phi <n.
\end{eqnarray*}%
Here, $I_{\Omega ,\phi }^{A,m}$ and $M_{\Omega ,\phi }^{A,m}$ are trivial
generalizations of the above commutators, respectively.

It is well known that the multilinear operators have been widely studied by
many authors. (For example, see \cite{Cohen, Gurbuz, Wu0} etc.) In 2013, Wu
and Lan \cite{Wu0} proved that $I_{\Omega ,\phi }^{A,m}$ and $M_{\Omega
	,\phi }^{A,m}$ are bounded from $L^{p\left( \cdot \right) }$ to $L^{q\left(
	\cdot \right) }$ for $D^{\gamma }A\in \dot{\Lambda}_{\beta }\left( {\mathbb{R%
}}^{n}\right) $. However, it is worth pointing out that so far Lipschitz
boundedness for $I_{\Omega ,\phi }^{A,m}$ and $M_{\Omega ,\phi }^{A,m}$ on
homogeneous variable exponent Herz type spaces has not been proved for $%
D^{\gamma }A\in \dot{\Lambda}_{\beta }\left( {\mathbb{R}}^{n}\right) $.

In this sense, we recall the definition of homogenous Lipschitz space $\dot{%
	\Lambda}_{\beta }\left( {\mathbb{R}}^{n}\right) $ as follows:

\begin{definition}
	$\left( \text{\textbf{Homogenous Lipschitz space}}\right) $ Let $0<\beta
	\leq 1$. The homogeneous Lipschitz space $\dot{\Lambda}_{\beta }$ is defined
	by%
	\begin{equation*}
		\dot{\Lambda}_{\beta }\left( {\mathbb{R}}^{n}\right) =\left \{ f:\left \Vert
		f\right \Vert _{\dot{\Lambda}_{\beta }}=\sup_{x,h\in 
			\mathbb{R}
			,h\neq 0}\frac{\left \vert \Delta _{h}^{\left[ \beta \right] +1}f\left(
			x\right) \right \vert }{\left \vert h\right \vert ^{\beta }}<\infty \right
		\} ,
	\end{equation*}%
	where $\Delta _{h}^{1}f\left( x\right) =f\left( x+h\right) -f\left( x\right) 
	$, $\Delta _{h}^{k+1}f\left( x\right) =\Delta _{h}^{k}f\left( x+h\right)
	-\Delta _{h}^{k}f\left( x\right) $, $k\geq 1$.
	
	Obviously, if $\beta >1$, then $\dot{\Lambda}_{\beta }\left( {\mathbb{R}}%
	^{n}\right) $ only includes constant. So we restrict $0<\beta \leq 1$ (see 
	\cite{Paluszynski} for details)
\end{definition}

Let us now give some necessary definitions and notations. Throughout this
work, $Q$ will denote a cube on $%
\mathbb{R}
^{n}$ with edges parallel to the axes. We will denote the cube with center $%
x_{0}$ and edge length $r$ by $Q=Q\left( x_{0},r\right) $. Given a cube $Q$
and $\delta >0$, we will denote the cube with center $Q$ and edge length $%
\delta $ times the edge length of $Q$ by $\delta Q$. For a cube $Q$, we use
the notation 
\begin{equation*}
	f_{Q}=\frac{1}{\left \vert Q\right \vert }\int \limits_{Q}f,
\end{equation*}%
where $f_{Q}$ is the center of $Q$.

Throughout this work, the constant $C>0$ may vary from step to another and
do depentent on parameters involved. The expression $f\lesssim g$ means $%
f\leqslant Cg$ and $f\thickapprox g$ implies that $f\lesssim g\lesssim f$.
Also, for simplicity, we denote $L^{p\left( \cdot \right) }\left( {\mathbb{R}%
	^{n}}\right) $ by $L^{p\left( \cdot \right) }$ and similarly $B(x,r)$ by $B$.

Some scholars found that as long as it is proved that the Hardy-Littlewood
maximal operator $\mathcal{M}$ is bounded on $L^{p\left( \cdot \right)
}\left( {\mathbb{R}^{n}}\right) $, many conclusions in the corresponding
classical harmonic analysis and function space theory can be established in
the corresponding variable exponent function space. In this context, we
denote by $\mathcal{P}\left( {\mathbb{R}^{n}}\right) $ the set of all
functions $p\left( \cdot \right) $ which are measurable and satisfy: $1\leq
p_{-}:=\limfunc{essinf}\limits_{x\in {\mathbb{R}^{n}}}p\left( x\right) $ and 
$p_{+}:=\limfunc{esssup}\limits_{x\in {\mathbb{R}^{n}}}p\left( x\right)
<\infty $.

Let $p\left( \cdot \right) \in \mathcal{P}\left( {\mathbb{R}^{n}}\right) $.
Variable exponent Lebesgue space $L^{p\left( \cdot \right) }\left( {\mathbb{R%
	}^{n}}\right) $ is defined by%
\begin{equation*}
	\left \Vert f\right \Vert _{L^{p\left( \cdot \right) }}:=\inf \left \{ \eta
	>0:\int \limits_{{\mathbb{R}^{n}}}\left( \frac{\left \vert f\left( x\right)
		\right \vert }{\eta }\right) ^{p\left( x\right) }dx\leq 1\right \} <\infty .
\end{equation*}

Let $f\in L_{loc}\left( {\mathbb{R}^{n}}\right) $. The Hardy-Littlewood
maximal operator $\mathcal{M}$ is defined by%
\begin{equation*}
	\mathcal{M}f\left( x\right) :=\sup_{r>0}r^{-n}\int \limits_{B}\left \vert
	f\left( y\right) \right \vert dy,\qquad \forall x\in {\mathbb{R}^{n},}
\end{equation*}%
where and follows $B=\left \{ y\in {\mathbb{R}^{n}:}\left \vert
x-y\right
\vert <r\right \} $ is the open ball centered at $x$ with radius $%
r$. $\mathcal{B}\left( {\mathbb{R}^{n}}\right) $ is the collection of $%
p\left( \cdot \right) \in \mathcal{P}\left( {\mathbb{R}^{n}}\right) $ that
satisfy the boundedness of $\mathcal{M}$ on $L^{p\left( \cdot \right)
}\left( {\mathbb{R}^{n}}\right) $ as $p\left( \cdot \right) \in \mathcal{B}%
\left( {\mathbb{R}^{n}}\right) $ in \cite{Capone}.

For all $x,y\in \mathbb{%
	\mathbb{R}
}^{n}$ and $C>0$, if $p\left( \cdot \right) \in \mathcal{P}\left( {\mathbb{R}%
	^{n}}\right) $ satisfies the requirement given below%
\begin{equation}
	\left \vert p\left( y\right) -p\left( x\right) \right \vert \leq \frac{-C}{%
		\ln \left( \left \vert x-y\right \vert \right) },\qquad \text{if }\left
	\vert x-y\right \vert \leq \frac{1}{2}  \label{11*}
\end{equation}%
\begin{equation}
	\left \vert p\left( y\right) -p\left( x\right) \right \vert \leq \frac{C}{%
		\ln \left( e+\left \vert x\right \vert \right) },\qquad \text{if }\left
	\vert x\right \vert \leq \left \vert y\right \vert ,  \label{12*}
\end{equation}%
then $p\left( \cdot \right) \in \mathcal{B}\left( {\mathbb{R}^{n}}\right) $
in \cite{Nekavinda}.

As we all know, H\"{o}lder's inequality is a very important tool in studying
the boundedness of operators. Of course, similar inequalities are also
needed in variable exponent function space, so there is a generalized H\"{o}%
lder's inequality.

For $p(\cdot )\in \mathcal{P}(\mathbb{R}^{n})$, $f\in L^{p\left( \cdot
	\right) }\left( {\mathbb{R}^{n}}\right) $ and $g\in L^{p^{\prime }\left(
	\cdot \right) }\left( {\mathbb{R}^{n}}\right) $, the integral form of H\"{o}%
lder's inequality in the context of variable exponent spaces takes the form 
\begin{equation}
	\int \limits_{%
		\mathbb{R}
		^{n}}\left \vert f\left( x\right) g\left( x\right) \right \vert dx\leq
	r_{p}\left \Vert f\right \Vert _{L^{p\left( \cdot \right) }}\left \Vert
	g\right \Vert _{L^{p^{\prime }\left( \cdot \right) }},  \label{3}
\end{equation}%
where the constant $r_{p}$ is given by%
\begin{equation*}
	r_{p}=1+\frac{1}{p_{-}}-\frac{1}{p_{+}},
\end{equation*}%
see Theorem 2.1 in \cite{Kovacik}.

On the other hand, Nakai and Sawano \cite{Nakai} defined another variable
exponent $\tilde{q}\left( \cdot \right) $ by%
\begin{equation*}
	\frac{1}{q}+\frac{1}{\tilde{q}\left( \cdot \right) }=\frac{1}{p\left( \cdot
		\right) }.
\end{equation*}%
Then, we have 
\begin{equation}
	\left \Vert f.g\right \Vert _{L^{p\left( \cdot \right) }}\lesssim \left
	\Vert f\right \Vert _{L^{\tilde{q}\left( \cdot \right) }}\left \Vert g\right
	\Vert _{L^{q}}  \label{31}
\end{equation}%
for $p(\cdot )\in \mathcal{P}(\mathbb{R}^{n})$, $\left( p\right) _{+}<q$ and
for all measurable functions $f$ and $g$.

Assume that $p(\cdot )\in \mathcal{P}(\mathbb{R}^{n})$ and satisfies (\ref%
{11*}) and (\ref{12*}). Then so does $p^{\prime }(\cdot )$. Generally, we
can observe that $p\left( \cdot \right) ,p^{\prime }\left( \cdot \right) \in 
\mathcal{B}\left( {\mathbb{R}^{n}}\right) $ from \cite{Nekavinda}..Thus, by
virtue of Lemma 1 in \cite{Izuki}, we can consider constants $\delta _{1}\in
\left( 0,\frac{1}{\left( p\right) _{+}}\right) $ and $\delta _{2}\in \left(
0,\frac{1}{\left( p^{\prime }\right) _{+}}\right) $ such that

\begin{equation}
	\frac{\left \Vert \chi _{S}\right \Vert _{L^{p\left( \cdot \right) }}}{\left
		\Vert \chi _{B}\right \Vert _{L^{p\left( \cdot \right) }}}\leq C\left( \frac{%
		\left \vert S\right \vert }{\left \vert B\right \vert }\right) ^{\delta
		_{1}},\frac{\left \Vert \chi _{S}\right \Vert _{L^{p^{\prime }\left( \cdot
				\right) }}}{\left \Vert \chi _{B}\right \Vert _{L^{p^{\prime }\left( \cdot
				\right) }}}\leq C\left( \frac{\left \vert S\right \vert }{\left \vert
		B\right \vert }\right) ^{\delta _{2}}  \label{1}
\end{equation}%
for $S\subset B$.

When $p\left( \cdot \right) \in \mathcal{B}\left( {\mathbb{R}^{n}}\right) $,
then{\large {\ }}%
\begin{equation}
	\Vert \chi _{B}\Vert _{L^{p(\cdot )}}\Vert \chi _{B}\Vert _{L^{p^{\prime
			}(\cdot )}}\leq C\left \vert B\right \vert  \label{5}
\end{equation}%
was proved in \cite{Izuki}.

Let $p(\cdot )\in \mathcal{P}(\mathbb{R}^{n})$ satisfying conditions (\ref%
{11*}) and (\ref{12*}), then

\begin{equation}
	\left \Vert \chi _{Q}\right \Vert _{L^{p\left( \cdot \right) }}\approx \left
	\{ 
	\begin{array}{c}
		\left \vert Q\right \vert ^{\frac{1}{p\left( x\right) }},\text{ if }\left
		\vert Q\right \vert \leq 2^{n}\text{ and }x\in Q, \\ 
		\left \vert Q\right \vert ^{\frac{1}{p\left( \infty \right) }},\text{ if }%
		\left \vert Q\right \vert >1,%
	\end{array}%
	\right.  \label{100}
\end{equation}%
for all cubes (balls) $Q\subset {\mathbb{R}^{n}}$, where $p\left( \infty
\right) =\lim \limits_{\left \vert x\right \vert \rightarrow \infty }p\left(
x\right) $ (see \cite{Diening}).

Let $l\in \mathbb{%
	\mathbb{Z}
}$, $B_{l}:=\{x\in \mathbb{R}^{n}:\left \vert x\right \vert \leq 2^{l}\}$, $%
\Delta _{l}:=B_{l}\setminus B_{l-1}$, $\chi _{l}:=\chi _{\Delta _{l}}$. For
any $m\in 
\mathbb{N}
_{0}=%
\mathbb{N}
\cup \left \{ 0\right \} $, we define

\begin{equation*}
	\tilde{\chi}_{m}:=\left \{ 
	\begin{array}{ccc}
		\chi _{\Delta _{m}} & , & m\geq 1 \\ 
		\chi _{B_{0}} & , & m=0%
	\end{array}%
	\right. \cdot
\end{equation*}%
By virtue of (\ref{1}), we obtain 
\begin{equation}
	\frac{\left \Vert \chi _{l}\right \Vert _{L^{p\left( \cdot \right) }}}{\left
		\Vert \chi _{B_{l}}\right \Vert _{L^{p\left( \cdot \right) }}}\leq C\left( 
	\frac{\left \vert \Delta _{l}\right \vert }{\left \vert B_{l}\right \vert }%
	\right) ^{\delta _{1}}\Longrightarrow \Vert \chi _{l}\Vert _{L^{p(\cdot
			)}}\lesssim \Vert \chi _{B_{l}}\Vert _{L^{p(\cdot )}}.  \label{32}
\end{equation}

\begin{definition}
	Let $\alpha \in 
	\mathbb{R}
	$, $0<q\leq \infty $ and $p\left( \cdot \right) \in \mathcal{P}\left( {%
		\mathbb{R}^{n}}\right) $. Then, the homogeneous variable exponent Herz space 
	$\dot{K}_{p(\cdot )}^{\alpha ,q}(\mathbb{R}^{n})$ is defined by%
	\begin{equation*}
		\dot{K}_{p(\cdot )}^{\alpha ,q}(\mathbb{R}^{n}):=\left \{ f\in
		L_{loc}^{p(\cdot )}\left( \mathbb{R}^{n}\setminus \{0\} \right) :\Vert
		f\Vert _{\dot{K}_{p(\cdot )}^{\alpha ,q}(\mathbb{R}^{n})}<\infty \right \} ,
	\end{equation*}%
	where%
	\begin{equation*}
		\Vert f\Vert _{\dot{K}_{p(\cdot )}^{\alpha ,q}(\mathbb{R}^{n})}:=\left( \sum
		\limits_{l=-\infty }^{\infty }\Vert 2^{l\alpha }f\chi _{l}\Vert _{L^{p(\cdot
				)}}^{q}\right) ^{\frac{1}{q}}
	\end{equation*}%
	with the usual modifications made when $q=\infty $.
\end{definition}

The first main result we wanted to find in this work is as follows.

\begin{theorem}
	\label{Theorem}Let $0<\phi <n$ and $\Omega $ be homogeneous of degree zero
	with $\Omega \in L^{s}(S^{n-1})\left( s>1\right) $. Let also $0<q_{1}\leq
	q_{2}<\infty $, $0<\beta <1$ and $\alpha \in 
	\mathbb{R}
	$ such that $\phi +\beta +n\delta _{2}<\alpha <$ $n\delta _{1}-\left( \phi
	+\beta +\frac{n-1}{s}\right) $ with $\delta _{1}$, $\delta _{2}\in \left(
	0,1\right) $ satisfying (\ref{1}). Assume that $p(\cdot ),p_{1}\left( \cdot
	\right) ,p_{2}\left( \cdot \right) \in \mathcal{P}(\mathbb{R}^{n})$ satisfy (%
	\ref{11*}) and (\ref{12*}) and define $p_{1}\left( \cdot \right) $ and $%
	p_{2}\left( \cdot \right) $ by $\frac{1}{p_{2}\left( \cdot \right) }=\frac{1%
	}{p_{1}\left( \cdot \right) }-\frac{\beta +\phi }{n}$. If $D^{\gamma }A\in 
	\dot{\Lambda}_{\beta }\left( {\mathbb{R}}^{n}\right) $ $\left( \left \vert
	\gamma \right \vert =m-1,m\geq 2\right) $ and $\left( p_{1}^{\prime }\right)
	_{+}<s$, then the following inequalities hold:%
	\begin{equation}
		\left \Vert I_{\Omega ,\phi }^{A,m}f\right \Vert _{\dot{K}_{p_{2}(\cdot
				)}^{\alpha ,q_{2}}(\mathbb{R}^{n})}\lesssim \sum \limits_{\left \vert
			\gamma \right \vert =m-1}\left \Vert D^{\gamma }A\right \Vert _{\dot{\Lambda}%
			_{\beta }\left( {\mathbb{R}}^{n}\right) }\left \Vert f\right \Vert _{\dot{K}%
			_{p_{1}(\cdot )}^{\alpha ,q_{1}}(\mathbb{R}^{n})},  \label{5*}
	\end{equation}%
	\begin{equation}
		\left \Vert M_{\Omega ,\phi }^{A,m}f\right \Vert _{\dot{K}_{p_{2}(\cdot
				)}^{\alpha ,q_{2}}(\mathbb{R}^{n})}\lesssim \sum \limits_{\left \vert
			\gamma \right \vert =m-1}\left \Vert D^{\gamma }A\right \Vert _{\dot{\Lambda}%
			_{\beta }\left( {\mathbb{R}}^{n}\right) }\left \Vert f\right \Vert _{\dot{K}%
			_{p_{1}(\cdot )}^{\alpha ,q_{1}}(\mathbb{R}^{n})}.  \label{6}
	\end{equation}
\end{theorem}

When $m=1$ in Theorem \ref{Theorem}, we have the following:

\begin{corollary}
	Under the conditions of Theorem \ref{Theorem}, the following boundedness
	estimates hold: 
	\begin{equation*}
		\left \Vert I_{\Omega ,\phi }^{A}f\right \Vert _{\dot{K}_{p_{2}(\cdot
				)}^{\alpha ,q_{2}}(\mathbb{R}^{n})}\lesssim \left \Vert A\right \Vert _{\dot{%
				\Lambda}_{\beta }\left( {\mathbb{R}}^{n}\right) }\left \Vert f\right \Vert _{%
			\dot{K}_{p_{1}(\cdot )}^{\alpha ,q_{1}}(\mathbb{R}^{n})},
	\end{equation*}%
	\begin{equation*}
		\left \Vert M_{\Omega ,\phi }^{A}f\right \Vert _{\dot{K}_{p_{2}(\cdot
				)}^{\alpha ,q_{2}}(\mathbb{R}^{n})}\lesssim \left \Vert A\right \Vert _{\dot{%
				\Lambda}_{\beta }\left( {\mathbb{R}}^{n}\right) }\left \Vert f\right \Vert _{%
			\dot{K}_{p_{1}(\cdot )}^{\alpha ,q_{1}}(\mathbb{R}^{n})}.
	\end{equation*}
\end{corollary}

\begin{definition}
	Let $\alpha \in 
	\mathbb{R}
	$, $0<q\leq \infty $, $p\left( \cdot \right) \in \mathcal{P}\left( {\mathbb{R%
		}^{n}}\right) $ and $0\leq \lambda <\infty $. Then, the homogeneous variable
	exponent Herz-Morrey space $M\dot{K}_{p(\cdot )}^{\alpha ,q}(\mathbb{R}^{n})$
	is defined by%
	\begin{equation*}
		M\dot{K}_{p(\cdot )}^{\alpha ,q}(\mathbb{R}^{n}):=\left \{ f\in
		L_{loc}^{p(\cdot )}\left( \mathbb{R}^{n}\setminus \{0\} \right) :\Vert
		f\Vert _{M\dot{K}_{p(\cdot )}^{\alpha ,q}(\mathbb{R}^{n})}<\infty \right \} ,
	\end{equation*}%
	where%
	\begin{equation*}
		\Vert f\Vert _{M\dot{K}_{p(\cdot )}^{\alpha ,q}(\mathbb{R}^{n})}:=\sup_{L\in 
			\mathbb{%
				\mathbb{Z}
		}}2^{-L\lambda }\left( \sum \limits_{l=-\infty }^{L}\Vert 2^{l\alpha }f\chi
		_{l}\Vert _{L^{p(\cdot )}}^{q}\right) ^{\frac{1}{q}}
	\end{equation*}%
	with the usual modifications made when $q=\infty $. Obviously, when $\lambda
	=0$, $M\dot{K}_{p(\cdot )}^{\alpha ,q}(\mathbb{R}^{n})=\dot{K}_{p(\cdot
		)}^{\alpha ,q}(\mathbb{R}^{n})$.
\end{definition}

Finally, the other main result we wanted to find in this article is as
follows.

\begin{theorem}
	\label{Theorem1}Let $0\leq \lambda <\infty $ such that $\phi +\beta +n\delta
	_{2}+\lambda <\alpha <$ $n\delta _{1}+\lambda -\left( \phi +\beta +\frac{n-1%
	}{s}\right) $ with $\delta _{1}$, $\delta _{2}\in \left( 0,1\right) $
	satisfying (\ref{1}) and under stipulations in Theorem \ref{Theorem}, the
	generalized commutators $I_{\Omega ,\phi }^{A,m}$ and $M_{\Omega ,\phi
	}^{A,m}$ satisfy%
	\begin{equation}
		\left \Vert I_{\Omega ,\phi }^{A,m}f\right \Vert _{M\dot{K}_{p_{2}(\cdot
				)}^{\alpha ,q_{2}}(\mathbb{R}^{n})}\lesssim \sum \limits_{\left \vert
			\gamma \right \vert =m-1}\left \Vert D^{\gamma }A\right \Vert _{\dot{\Lambda}%
			_{\beta }\left( {\mathbb{R}}^{n}\right) }\left \Vert f\right \Vert _{M\dot{K}%
			_{p_{1}(\cdot )}^{\alpha ,q_{1}}(\mathbb{R}^{n})},  \label{7}
	\end{equation}%
	\begin{equation}
		\left \Vert M_{\Omega ,\phi }^{A,m}f\right \Vert _{M\dot{K}_{p_{2}(\cdot
				)}^{\alpha ,q_{2}}(\mathbb{R}^{n})}\lesssim \sum \limits_{\left \vert
			\gamma \right \vert =m-1}\left \Vert D^{\gamma }A\right \Vert _{\dot{\Lambda}%
			_{\beta }\left( {\mathbb{R}}^{n}\right) }\left \Vert f\right \Vert _{M\dot{K}%
			_{p_{1}(\cdot )}^{\alpha ,q_{1}}(\mathbb{R}^{n})}.  \label{8}
	\end{equation}
\end{theorem}

For $m=1$ in Theorem \ref{Theorem1}, we get

\begin{corollary}
	Let $0\leq \lambda <\infty $ such that $\phi +\beta +n\delta _{2}+\lambda
	<\alpha <$ $n\delta _{1}+\lambda -\left( \phi +\beta +\frac{n-1}{s}\right) $
	with $\delta _{1}$, $\delta _{2}\in \left( 0,1\right) $ satisfying (\ref{1})
	and under the conditions of Theorem \ref{Theorem}, the following fundamental
	inequalities hold:%
	\begin{equation*}
		\left \Vert I_{\Omega ,\phi }^{A}f\right \Vert _{M\dot{K}_{p_{2}(\cdot
				)}^{\alpha ,q_{2}}(\mathbb{R}^{n})}\lesssim \left \Vert A\right \Vert _{\dot{%
				\Lambda}_{\beta }\left( {\mathbb{R}}^{n}\right) }\left \Vert f\right \Vert
		_{M\dot{K}_{p_{1}(\cdot )}^{\alpha ,q_{1}}(\mathbb{R}^{n})},
	\end{equation*}%
	\begin{equation*}
		\left \Vert M_{\Omega ,\phi }^{A}f\right \Vert _{M\dot{K}_{p_{2}(\cdot
				)}^{\alpha ,q_{2}}(\mathbb{R}^{n})}\lesssim \left \Vert A\right \Vert _{\dot{%
				\Lambda}_{\beta }\left( {\mathbb{R}}^{n}\right) }\left \Vert f\right \Vert
		_{M\dot{K}_{p_{1}(\cdot )}^{\alpha ,q_{1}}(\mathbb{R}^{n})}.
	\end{equation*}
\end{corollary}

\section{Main Lemmas}

Before proving Theorems \ref{Theorem} and \ref{Theorem1}, following lemmas
are needed. These lemmas will be helpful in proving main results.

\begin{lemma}
	\label{Lemma1}$\left( \text{see \cite{Cohen}}\right) $ Let $A\left( x\right) 
	$ be a function defined on ${\mathbb{R}^{n}}$ with $m$-th order derivatives
	on $L_{loc}^{q}\left( {\mathbb{R}^{n}}\right) $ for $\left \vert \gamma
	\right \vert =m$ and some $q>n$. Then,%
	\begin{equation*}
		\left \vert R_{m}\left( A;x,y\right) \right \vert \leq C\left \vert
		x-y\right \vert ^{m}\sum \limits_{\left \vert \gamma \right \vert =m}\left( 
		\frac{1}{\left \vert \tilde{Q}\right \vert }\int \limits_{\tilde{Q}}\left
		\vert D^{\gamma }A\left( z\right) \right \vert ^{q}dz\right) ^{\frac{1}{q}},
	\end{equation*}%
	where $\tilde{Q}$ is the cube centered at $x$ and having diameter $5\sqrt{n}%
	\left \vert x-y\right \vert $.
\end{lemma}

\begin{lemma}
	\label{Lemma1*}$\left( \text{see \cite{Paluszynski}}\right) $ For $0<\beta
	<1 $, $1\leq q<\infty $, we have%
	\begin{eqnarray*}
		\left \Vert f\right \Vert _{\dot{\Lambda}_{\beta }\left( {\mathbb{R}}%
			^{n}\right) } &=&\sup_{Q}\frac{1}{\left \vert Q\right \vert ^{1+\frac{\beta 
				}{n}}}\int \limits_{Q}\left \vert f-f_{Q}\right \vert \\
		&\approx &\sup_{Q}\frac{1}{\left \vert Q\right \vert ^{\frac{\beta }{n}}}%
		\left( \frac{1}{\left \vert Q\right \vert }\int \limits_{Q}\left \vert
		f-f_{Q}\right \vert ^{q}dx\right) ^{\frac{1}{q}}.
	\end{eqnarray*}
\end{lemma}

\begin{lemma}
	\label{Lemma2}$\left( \text{see \cite{Paluszynski}}\right) $ Let $Q^{\ast
	}\subset Q$. If $f\in \dot{\Lambda}_{\beta }\left( {\mathbb{R}}^{n}\right)
	\left( 0<\beta <1\right) $, then%
	\begin{equation*}
		\left \vert f_{Q^{\ast }}-f_{Q}\right \vert \leq C\left \Vert f\right \Vert
		_{\dot{\Lambda}_{\beta }\left( {\mathbb{R}}^{n}\right) }\left \vert Q\right
		\vert ^{\frac{\beta }{n}}.
	\end{equation*}
\end{lemma}

In the proof of our main results we need to obtain an local estimate of\ $%
\left \vert R_{m}\left( A;x,y\right) \right \vert $, similar to the Lemma %
\ref{Lemma1}. To do this, applying the above results, we obtain the
following:

\begin{lemma}
	Let $0<\beta <1$, $A\left( x\right) $ be a function defined on ${\mathbb{R}%
		^{n}}$ with $m$-th order derivatives on $L_{loc}^{q}\left( {\mathbb{R}^{n}}%
	\right) $ for $\left \vert \gamma \right \vert =m$ and some $q>n$. Then,%
	\begin{equation}
		\left \vert R_{m}\left( A;x,y\right) \right \vert \lesssim \sum
		\limits_{\left \vert \gamma \right \vert =m-1}\left \Vert D^{\gamma }A\right
		\Vert _{\dot{\Lambda}_{\beta }\left( {\mathbb{R}}^{n}\right) }\left \vert
		x-y\right \vert ^{m-1+\beta }.  \label{21}
	\end{equation}
\end{lemma}

\begin{proof}
	For any $x\in {\mathbb{R}^{n}}$, let $Q\left( x,y\right) =Q$ be the cube
	centered at $x$ and having diameter $4\sqrt{n}\left \vert x-y\right \vert $.
	For fixed $x\in {\mathbb{R}^{n}}$, let%
	\begin{equation*}
		A_{Q}\left( y\right) =A\left( y\right) -\sum \limits_{\left \vert \gamma
			\right \vert =m-1}\frac{1}{\gamma !}\left( D^{\gamma }A\right) _{Q}y^{\gamma
		},
	\end{equation*}%
	where $\left( D^{\gamma }A\right) _{Q}$ is the average of $D^{\gamma }A$ on $%
	Q$. From the definitions of $R_{m}\left( A;x,y\right) $ and $A_{Q}$, it is
	easy to see that $R_{m}\left( A;x,y\right) =R_{m}\left( A_{Q};x,y\right) $
	and%
	\begin{equation*}
		R_{m}\left( A_{Q};x,y\right) =R_{m-1}\left( A_{Q};x,y\right) -\sum
		\limits_{\left \vert \gamma \right \vert =m-1}\frac{1}{\gamma !}D^{\gamma
		}A_{Q}\left( x\right) \left \vert x-y\right \vert ^{\gamma }\text{.}
	\end{equation*}%
	Then, 
	\begin{equation}
		\left \vert R_{m}\left( A;x,y\right) \right \vert \lesssim \left \vert
		R_{m-1}\left( A_{Q};x,y\right) \right \vert +\sum \limits_{\left \vert
			\gamma \right \vert =m-1}\frac{1}{\gamma !}\left \vert D^{\gamma
		}A_{Q}\left( x\right) \right \vert \left \vert x-y\right \vert ^{m-1}.
		\label{22}
	\end{equation}%
	Applying Lemmas \ref{Lemma1}, \ref{Lemma1*} and \ref{Lemma2}, for any $y$,
	we get%
	\begin{eqnarray}
		\left \vert R_{m-1}\left( A_{Q};x,y\right) \right \vert &\lesssim &\left
		\vert x-y\right \vert ^{m-1}\sum \limits_{\left \vert \gamma \right \vert
			=m-1}\left( \frac{1}{\left \vert \tilde{Q}\right \vert }\int \limits_{%
			\tilde{Q}}\left \vert D^{\gamma }A_{Q}\left( z\right) \right \vert
		^{q}dy\right) ^{\frac{1}{q}}  \notag \\
		&=&\left \vert x-y\right \vert ^{m-1}\sum \limits_{\left \vert \gamma
			\right \vert =m-1}\left( \frac{1}{\left \vert \tilde{Q}\right \vert }\int
		\limits_{\tilde{Q}}\left \vert D^{\gamma }A\left( z\right) -\left( D^{\gamma
		}A\right) _{Q}\right \vert ^{q}dy\right) ^{\frac{1}{q}}  \notag \\
		&\lesssim &\left \vert x-y\right \vert ^{m-1}\left[ 
		\begin{array}{c}
			\left( \frac{1}{\left \vert \tilde{Q}\right \vert }\int \limits_{\tilde{Q}%
			}\left \vert D^{\gamma }A\left( z\right) -\left( D^{\gamma }A\right) _{%
				\tilde{Q}}\right \vert ^{q}dy\right) ^{\frac{1}{q}} \\ 
			+\left \vert \left( D^{\gamma }A\right) _{\tilde{Q}}-\left( D^{\gamma
			}A\right) _{5Q}\right \vert +\left \vert \left( D^{\gamma }A\right)
			_{5Q}-\left( D^{\gamma }A\right) _{Q}\right \vert%
		\end{array}%
		\right]  \notag \\
		&\lesssim &\left \vert x-y\right \vert ^{m+\beta -1}\left \Vert D^{\gamma
		}A\right \Vert _{\dot{\Lambda}_{\beta }\left( {\mathbb{R}}^{n}\right) },
		\label{23}
	\end{eqnarray}%
	where $\tilde{Q}$ is the cube centered at $x$ and having diameter $5\sqrt{n}%
	\left \vert x-y\right \vert $ and $\tilde{Q}\subset 5Q$.
	
	On the other hand, applying Lemma \ref{Lemma2} again, we have 
	\begin{equation}
		\left \vert D^{\gamma }A_{Q}\left( x\right) \right \vert \lesssim \left
		\Vert D^{\gamma }A\right \Vert _{\dot{\Lambda}_{\beta }\left( {\mathbb{R}}%
			^{n}\right) }\left \vert Q\right \vert ^{\frac{\beta }{n}}\lesssim \left
		\Vert D^{\gamma }A\right \Vert _{\dot{\Lambda}_{\beta }\left( {\mathbb{R}}%
			^{n}\right) }\left \vert x-y\right \vert ^{\beta }.  \label{24}
	\end{equation}%
	Thus, combining inequalities (\ref{22})-(\ref{24}), we obtain (\ref{21}).
\end{proof}

\begin{lemma}
	\label{Lemma3}Let $\Omega \left( x\right) $, $\Delta _{k}$, $\Delta _{z}$ $%
	\left( k,z\in 
	\mathbb{Z}
	\right) $ be as stated above and we have
	
	$\left( i\right) $ if $x\in \Delta _{k}$, $k\leq z-3$, then 
	\begin{equation*}
		\left( \int \limits_{\Delta _{z}}\left \vert \Omega (x-y)\right \vert
		^{s}dy\right) ^{\frac{1}{s}}\lesssim 2^{\frac{zn}{s}}\left \Vert \Omega
		\right \Vert _{L^{s}(S^{n-1})},
	\end{equation*}
	
	$\left( ii\right) $ if $x\in \Delta _{k}$, $k\geq z+3$, then 
	\begin{equation*}
		\left( \int \limits_{\Delta _{z}}\left \vert \Omega (x-y)\right \vert
		^{s}dy\right) ^{\frac{1}{s}}\lesssim 2^{\frac{\left( z-k+kn\right) }{s}%
		}\left \Vert \Omega \right \Vert _{L^{s}(S^{n-1})}.
	\end{equation*}
\end{lemma}

\begin{proof}
	When $x\in \Delta _{k}$, $y\in \Delta _{z}$, $\left \vert k-z\right \vert
	\geq 3$, then we have $\left \vert x\right \vert \approx 2^{k}$, $%
	\left
	\vert y\right \vert \approx 2^{z}$, $\left \vert x-y\right \vert
	\approx \max \left \{ 2^{z},2^{k}\right \} $ and%
	\begin{equation*}
		\left( \int \limits_{\Delta _{z}}\left \vert \Omega (x-y)\right \vert
		^{s}dy\right) ^{\frac{1}{s}}\leq \left( \int \limits_{x+B_{z}}\left \vert
		\Omega (x-y)\right \vert ^{s}dy\right) ^{\frac{1}{s}}.
	\end{equation*}%
	Thus, direct calculation yields $\left( i\right) $ and $\left( ii\right) $
	(see \cite{Wu} for more details).
\end{proof}

\section{Theorem Proofs}

\textbf{Proof of Theorem \ref{Theorem}.}

\begin{proof}
	Let $f\in \dot{K}_{p_{1}(\cdot )}^{\alpha ,q_{1}}(\mathbb{R}^{n})$. Taking $%
	f_{z}:=f\chi _{z}$ for each $z\in 
	\mathbb{Z}
	$, we write 
	\begin{equation*}
		f\left( x\right) =\sum \limits_{z=-\infty }^{\infty }f\left( x\right) \chi
		_{z}\left( x\right) =\sum \limits_{z=-\infty }^{\infty }f_{z}\left(
		x\right) .
	\end{equation*}%
	Owing to $0<\frac{q_{1}}{q_{2}}\leq 1$, then the Jensen inequality is:%
	\begin{equation}
		\left( \sum \limits_{j=-\infty }^{\infty }\left \vert c_{j}\right \vert
		\right) ^{\frac{q_{1}}{q_{2}}}\leq \sum \limits_{j=-\infty }^{\infty }\left
		\vert c_{j}\right \vert ^{\frac{q_{1}}{q_{2}}},c_{j}\in 
		\mathbb{R}
		,j\in 
		\mathbb{Z}
		.  \label{25}
	\end{equation}%
	By (\ref{25}), we have 
	\begin{eqnarray}
		\left \Vert I_{\Omega ,\phi }^{A,m}f\right \Vert _{\dot{K}_{p_{2}(\cdot
				)}^{\alpha ,q_{2}}(\mathbb{R}^{n})}^{q_{1}} &=&\left( \sum
		\limits_{k=-\infty }^{\infty }2^{k\alpha q_{2}}\Vert \left( I_{\Omega ,\phi
		}^{A,m}f\right) \chi _{k}\Vert _{L^{p_{2}(\cdot )}}^{q_{2}}\right) ^{\frac{%
				q_{1}}{q_{2}}}  \notag \\
		&\lesssim &\sum \limits_{k=-\infty }^{\infty }2^{k\alpha q_{1}}\Vert \left(
		I_{\Omega ,\phi }^{A,m}f\right) \chi _{k}\Vert _{L^{p_{2}(\cdot )}}^{q_{1}} 
		\notag \\
		&\lesssim &\sum \limits_{k=-\infty }^{\infty }2^{k\alpha q_{1}}\left( \sum
		\limits_{z=-\infty }^{k-3}\Vert \left( I_{\Omega ,\phi }^{A,m}f_{z}\right)
		\chi _{k}\Vert _{L^{p_{2}(\cdot )}}\right) ^{q_{1}}  \notag \\
		&&+\sum \limits_{k=-\infty }^{\infty }2^{k\alpha q_{1}}\left( \sum
		\limits_{z=k-2}^{k+2}\Vert \left( I_{\Omega ,\phi }^{A,m}f_{z}\right) \chi
		_{k}\Vert _{L^{p_{2}(\cdot )}}\right) ^{q_{1}}  \notag \\
		&&+\sum \limits_{k=-\infty }^{\infty }2^{k\alpha q_{1}}\left( \sum
		\limits_{z=k+3}^{\infty }\Vert \left( I_{\Omega ,\phi }^{A,m}f_{z}\right)
		\chi _{k}\Vert _{L^{p_{2}(\cdot )}}\right) ^{q_{1}}  \notag \\
		&=&:X+Y+Z.  \label{101}
	\end{eqnarray}
	
	We first estimate $X$.
	
	Let $z\in \mathbb{%
		\mathbb{Z}
	}$, $B_{z}:=\{x\in \mathbb{R}^{n}:\left \vert x\right \vert \leq 2^{z}\}$, $%
	\Delta _{z}:=B_{z}\setminus B_{z-1}$. For any $k,z\in 
	\mathbb{Z}
	$ and $k\geq z+3$, using (\ref{3}) and (\ref{21}), we have%
	\begin{eqnarray*}
		\left \vert \left( I_{\Omega ,\phi }^{A,m}f_{z}\right) \chi _{k}\right \vert
		&\lesssim &\int \limits_{\Delta _{z}}\frac{\left \vert \Omega (x-y)\right
			\vert \left \vert f_{z}(y)\right \vert }{|x-y|^{n-\phi +m-1}}\left \vert
		R_{m}\left( A;x,y\right) \right \vert dy\cdot \chi _{k} \\
		&\lesssim &2^{k\left( \phi +\beta -n\right) }\sum \limits_{\left \vert
			\gamma \right \vert =m-1}\left \Vert D^{\gamma }A\right \Vert _{\dot{\Lambda}%
			_{\beta }\left( {\mathbb{R}}^{n}\right) }\int \limits_{%
			\mathbb{R}
			^{n}}\left \vert \Omega (x-y)\right \vert \left \vert f_{z}(y)\right \vert
		dy\cdot \chi _{k} \\
		&\lesssim &2^{k\left( \phi +\beta -n\right) }\sum \limits_{\left \vert
			\gamma \right \vert =m-1}\left \Vert D^{\gamma }A\right \Vert _{\dot{\Lambda}%
			_{\beta }\left( {\mathbb{R}}^{n}\right) }\left \Vert f_{z}\right \Vert
		_{L^{p_{1}(\cdot )}}\left \Vert \Omega (x-y)\chi _{z}\right \Vert
		_{L^{p_{^{1}}^{\prime }\left( \cdot \right) }}\cdot \chi _{k}
	\end{eqnarray*}%
	Since $\left( p_{1}^{\prime }\right) _{+}<s$ and $\frac{1}{s}+\frac{1}{%
		\tilde{p}_{1}^{\prime }\left( \cdot \right) }=\frac{1}{p_{1}^{\prime }\left(
		\cdot \right) }$, then using (\ref{31}), (\ref{32}) and Lemma \ref{Lemma3}
	give the following%
	\begin{eqnarray*}
		\left \vert \left( I_{\Omega ,\phi }^{A,m}f_{z}\right) \chi _{k}\right \vert
		&\lesssim &2^{k\left( \phi +\beta -n\right) }\sum \limits_{\left \vert
			\gamma \right \vert =m-1}\left \Vert D^{\gamma }A\right \Vert _{\dot{\Lambda}%
			_{\beta }\left( {\mathbb{R}}^{n}\right) }\left \Vert f_{z}\right \Vert
		_{L^{p_{1}(\cdot )}}\left \Vert \Omega (x-y)\chi _{z}\right \Vert
		_{L^{s}}\left \Vert \chi _{z}\right \Vert _{L^{\tilde{p}_{1}^{\prime }\left(
				\cdot \right) }}\cdot \chi _{k} \\
		&\lesssim &2^{k\left( \phi +\beta -n\right) }2^{\frac{\left( z-k+kn\right) }{%
				s}}\sum \limits_{\left \vert \gamma \right \vert =m-1}\left \Vert D^{\gamma
		}A\right \Vert _{\dot{\Lambda}_{\beta }\left( {\mathbb{R}}^{n}\right) }\left
		\Vert f_{z}\right \Vert _{L^{p_{1}(\cdot )}}\left \Vert \chi _{B_{z}}\right
		\Vert _{L^{\tilde{p}_{1}^{\prime }\left( \cdot \right) }}\cdot \chi _{k}
	\end{eqnarray*}%
	and%
	\begin{equation}
		\left \Vert \left( I_{\Omega ,\phi }^{A,m}f_{z}\right) \chi _{k}\right \Vert
		_{L^{p_{2}\left( \cdot \right) }}\lesssim 2^{k\left( \phi +\beta -n\right)
		}2^{\frac{\left( z-k+kn\right) }{s}}\sum \limits_{\left \vert \gamma \right
			\vert =m-1}\left \Vert D^{\gamma }A\right \Vert _{\dot{\Lambda}_{\beta
			}\left( {\mathbb{R}}^{n}\right) }\left \Vert f_{z}\right \Vert
		_{L^{p_{1}(\cdot )}}\left \Vert \chi _{B_{z}}\right \Vert _{L^{\tilde{p}%
				_{1}^{\prime }\left( \cdot \right) }}\left \Vert \chi _{k}\right \Vert
		_{L^{p_{2}\left( \cdot \right) }}.  \label{26}
	\end{equation}%
	According to (\ref{100}),
	
	$1-$ When $\left \vert B_{z}\right \vert \leq 2^{n}$ and $x_{z}\in B_{z}$,
	we have%
	\begin{equation*}
		\left \Vert \chi _{B_{z}}\right \Vert _{L^{\tilde{p}_{1}^{\prime }\left(
				\cdot \right) }}\approx \left \vert B_{z}\right \vert ^{\frac{1}{\tilde{p}%
				_{1}^{\prime }\left( x_{z}\right) }}\approx \left \Vert \chi _{B_{z}}\right
		\Vert _{L^{p_{^{1}}^{\prime }\left( \cdot \right) }}\left \vert B_{z}\right
		\vert ^{-\frac{1}{s}}.
	\end{equation*}
	
	$2-$ When $\left \vert B_{z}\right \vert \geq 1$, we have%
	\begin{equation*}
		\left \Vert \chi _{B_{z}}\right \Vert _{L^{\tilde{p}_{1}^{\prime }\left(
				\cdot \right) }}\approx \left \vert B_{z}\right \vert ^{\frac{1}{\tilde{p}%
				_{1}^{\prime }\left( \infty \right) }}\approx \left \Vert \chi
		_{B_{z}}\right \Vert _{L^{p_{^{1}}^{\prime }\left( \cdot \right) }}\left
		\vert B_{z}\right \vert ^{-\frac{1}{s}}.
	\end{equation*}%
	So, we get 
	\begin{equation}
		\left \Vert \chi _{B_{z}}\right \Vert _{L^{\tilde{p}_{1}^{\prime }\left(
				\cdot \right) }}\approx \left \Vert \chi _{B_{z}}\right \Vert
		_{L^{p_{^{1}}^{\prime }\left( \cdot \right) }}\left \vert B_{z}\right \vert
		^{-\frac{1}{s}}.  \label{27}
	\end{equation}%
	Thus, using inequality (\ref{27}) in (\ref{26}), we get%
	\begin{equation}
		\left \Vert \left( I_{\Omega ,\phi }^{A,m}f_{z}\right) \chi _{k}\right \Vert
		_{L^{p_{2}\left( \cdot \right) }}\lesssim 2^{k\left( \phi +\beta -n\right)
		}2^{\frac{\left( k-z\right) \left( n-1\right) }{s}}\sum \limits_{\left
			\vert \gamma \right \vert =m-1}\left \Vert D^{\gamma }A\right \Vert _{\dot{%
				\Lambda}_{\beta }\left( {\mathbb{R}}^{n}\right) }\left \Vert f_{z}\right
		\Vert _{L^{p_{1}(\cdot )}}\left \Vert \chi _{B_{z}}\right \Vert
		_{L^{p_{^{1}}^{\prime }\left( \cdot \right) }}\left \Vert \chi _{k}\right
		\Vert _{L^{p_{2}\left( \cdot \right) }}.  \label{27*}
	\end{equation}%
	For the establishment of the boundedness of the generalized commutators, we
	use the fundamental properties of the Riesz potential, which is formally
	defined as%
	\begin{equation*}
		I_{\phi }f(x)=\int \limits_{%
			\mathbb{R}
			^{n}}\frac{f(y)}{|x-y|^{n-\phi }}dy\qquad 0<\phi <n.
	\end{equation*}%
	Since%
	\begin{equation*}
		2^{\phi z}\chi _{B_{z}}\left( x\right) \lesssim \int \limits_{B_{z}}\frac{dy%
		}{|x-y|^{n-\phi }}\cdot \chi _{B_{z}}\left( x\right) \leq I_{\phi }\left(
		\chi _{B_{z}}\right) \left( x\right) ,
	\end{equation*}%
	then%
	\begin{equation}
		2^{z\left( \phi +\beta \right) }\chi _{B_{z}}\lesssim I_{\phi +\beta }\left(
		\chi _{B_{z}}\right) \left( x\right) \chi _{B_{z}}\leq I_{\phi +\beta
		}\left( \chi _{B_{z}}\right) \left( x\right)  \label{28}
	\end{equation}%
	(see \cite{Gurbuz1} p. 6). On the other hand, using (\ref{5}) and (\ref{32}%
	), we have%
	\begin{eqnarray}
		2^{k\left( \phi +\beta -n\right) }\left \Vert \chi _{B_{z}}\right \Vert
		_{L^{p_{^{1}}^{\prime }\left( \cdot \right) }}\left \Vert \chi
		_{B_{k}}\right \Vert _{L^{p_{2}\left( \cdot \right) }} &\lesssim &2^{k\left(
			\phi +\beta \right) }\left \Vert \chi _{B_{z}}\right \Vert
		_{L^{p_{^{1}}^{\prime }\left( \cdot \right) }}\left \Vert \chi
		_{B_{k}}\right \Vert _{L^{p_{2}^{\prime }\left( \cdot \right) }}^{-1}  \notag
		\\
		&\lesssim &2^{k\left( \phi +\beta \right) }\left \Vert \chi _{B_{z}}\right
		\Vert _{L^{p_{^{1}}^{\prime }\left( \cdot \right) }}\left \Vert \chi
		_{B_{z}}\right \Vert _{L^{p_{2}^{\prime }\left( \cdot \right)
		}}^{-1}2^{n\delta _{1}\left( z-k\right) }.  \label{29}
	\end{eqnarray}%
	Next, applying (\ref{3}), (\ref{5}), (\ref{28}) and from $\left(
	L^{p_{1}\left( \cdot \right) },L^{p_{2}\left( \cdot \right) }\right) $%
	-boundedness of\ $I_{\phi +\beta }$\ (see (3.5) in \cite{Gurbuz1}), we know
	that%
	\begin{eqnarray}
		\left \Vert \chi _{B_{z}}\right \Vert _{L^{p_{2}^{\prime }\left( \cdot
				\right) }}^{-1} &\lesssim &2^{-nz}\left \Vert \chi _{B_{z}}\right \Vert
		_{L^{p_{2}\left( \cdot \right) }}  \notag \\
		&\lesssim &2^{-nz}2^{-z\left( \phi +\beta \right) }\left \Vert I_{\phi
			+\beta }\left( \chi _{B_{z}}\right) \right \Vert _{L^{p_{2}\left( \cdot
				\right) }}  \notag \\
		&\lesssim &2^{-nz}2^{-z\left( \phi +\beta \right) }\left \Vert \chi
		_{B_{z}}\right \Vert _{L^{p_{1}\left( \cdot \right) }}  \notag \\
		&\lesssim &2^{-z\left( \phi +\beta \right) }\left \Vert \chi _{B_{z}}\right
		\Vert _{L^{p_{^{1}}^{\prime }\left( \cdot \right) }}^{-1}.  \label{30}
	\end{eqnarray}%
	Thus, using inequalities (\ref{29}) and (\ref{30}) in (\ref{27*}), we get%
	\begin{equation}
		\left \Vert \left( I_{\Omega ,\phi }^{A,m}f_{z}\right) \chi _{k}\right \Vert
		_{L^{p_{2}\left( \cdot \right) }}\lesssim 2^{\left( k-z\right) \left( \phi
			+\beta -n\delta _{1}+\frac{n-1}{s}\right) }\sum \limits_{\left \vert \gamma
			\right \vert =m-1}\left \Vert D^{\gamma }A\right \Vert _{\dot{\Lambda}%
			_{\beta }\left( {\mathbb{R}}^{n}\right) }\left \Vert f_{z}\right \Vert
		_{L^{p_{1}(\cdot )}}.  \label{37}
	\end{equation}%
	Thus, by virtue of (\ref{37}) and remark that $\alpha <$ $n\delta
	_{1}-\left( \phi +\beta +\frac{n-1}{s}\right) $,%
	\begin{eqnarray}
		X &=&\sum \limits_{k=-\infty }^{\infty }2^{k\alpha q_{1}}\left( \sum
		\limits_{z=-\infty }^{k-3}\Vert \left( I_{\Omega ,\phi }^{A,m}f_{z}\right)
		\chi _{k}\Vert _{L^{p_{2}(\cdot )}}\right) ^{q_{1}}  \notag \\
		&\lesssim &\sum \limits_{k=-\infty }^{\infty }2^{k\alpha q_{1}}\left( \sum
		\limits_{z=-\infty }^{k-3}2^{\left( k-z\right) \left( \phi +\beta -n\delta
			_{1}+\frac{n-1}{s}\right) }\sum \limits_{\left \vert \gamma \right \vert
			=m-1}\left \Vert D^{\gamma }A\right \Vert _{\dot{\Lambda}_{\beta }\left( {%
				\mathbb{R}}^{n}\right) }\left \Vert f_{z}\right \Vert _{L^{p_{1}(\cdot
				)}}\right) ^{q_{1}}  \notag \\
		&\lesssim &\sum \limits_{\left \vert \gamma \right \vert =m-1}\left \Vert
		D^{\gamma }A\right \Vert _{\dot{\Lambda}_{\beta }\left( {\mathbb{R}}%
			^{n}\right) }^{q_{1}}\sum \limits_{k=-\infty }^{\infty }\left( \sum
		\limits_{z=-\infty }^{k-3}2^{\alpha z}\left \Vert f_{z}\right \Vert
		_{L^{p_{1}(\cdot )}}2^{\left( k-z\right) \left( \phi +\beta -n\delta
			_{1}+\alpha +\frac{n-1}{s}\right) }\right) ^{q_{1}}.  \notag
	\end{eqnarray}%
	To continue estimating $X$, we split the problem into the following two
	cases:
	
	\textbf{Case 1 }$\left( 0<q_{1}\leq 1\right) .$
	
	Using (\ref{25}) and substituting $q_{1}$ for $\frac{q_{1}}{q_{2}}$, we
	obtain%
	\begin{eqnarray*}
		X &\lesssim &\sum \limits_{\left \vert \gamma \right \vert =m-1}\left \Vert
		D^{\gamma }A\right \Vert _{\dot{\Lambda}_{\beta }\left( {\mathbb{R}}%
			^{n}\right) }^{q_{1}}\sum \limits_{k=-\infty }^{\infty }\sum
		\limits_{z=-\infty }^{k-3}2^{\alpha zq_{1}}\left \Vert f_{z}\right \Vert
		_{L^{p_{1}(\cdot )}}^{q_{1}}2^{\left( k-z\right) \left( \phi +\beta -n\delta
			_{1}+\alpha +\frac{n-1}{s}\right) q_{1}} \\
		&\lesssim &\sum \limits_{\left \vert \gamma \right \vert =m-1}\left \Vert
		D^{\gamma }A\right \Vert _{\dot{\Lambda}_{\beta }\left( {\mathbb{R}}%
			^{n}\right) }^{q_{1}}\left \Vert f\right \Vert _{\dot{K}_{p_{1}(\cdot
				)}^{\alpha ,q_{1}}(\mathbb{R}^{n})}^{q_{1}}.
	\end{eqnarray*}%
	\textbf{Case 2 }$\left( 1<q_{1}<\infty \right) .$
	
	Let $\frac{1}{q_{1}}+\frac{1}{q_{1}^{\prime }}=1$. By (\ref{3}), we get%
	\begin{eqnarray*}
		X &\lesssim &\sum \limits_{\left \vert \gamma \right \vert =m-1}\left \Vert
		D^{\gamma }A\right \Vert _{\dot{\Lambda}_{\beta }\left( {\mathbb{R}}%
			^{n}\right) }^{q_{1}}\sum \limits_{k=-\infty }^{\infty }\sum
		\limits_{z=-\infty }^{k-3}2^{\alpha zq_{1}}\left \Vert f_{z}\right \Vert
		_{L^{p_{1}(\cdot )}}^{q_{1}}2^{\left( k-z\right) \left( \phi +\beta -n\delta
			_{1}+\alpha +\frac{n-1}{s}\right) \frac{q_{1}}{2}} \\
		&&\times \left( \sum \limits_{z=-\infty }^{k-3}2^{\left( k-z\right) \left(
			\phi +\beta -n\delta _{1}+\alpha +\frac{n-1}{s}\right) \frac{q_{1}^{\prime }%
			}{2}}\right) ^{\frac{q_{1}}{q_{1}^{\prime }}}
	\end{eqnarray*}%
	\begin{eqnarray*}
		&\lesssim &\sum \limits_{\left \vert \gamma \right \vert =m-1}\left \Vert
		D^{\gamma }A\right \Vert _{\dot{\Lambda}_{\beta }\left( {\mathbb{R}}%
			^{n}\right) }^{q_{1}}\sum \limits_{z=-\infty }^{\infty }2^{\alpha
			zq_{1}}\left \Vert f_{z}\right \Vert _{L^{p_{1}(\cdot )}}^{q_{1}}\sum
		\limits_{k=z+3}^{\infty }2^{\left( k-z\right) \left( \phi +\beta -n\delta
			_{1}+\alpha +\frac{n-1}{s}\right) \frac{q_{1}}{2}} \\
		&\lesssim &\sum \limits_{\left \vert \gamma \right \vert =m-1}\left \Vert
		D^{\gamma }A\right \Vert _{\dot{\Lambda}_{\beta }\left( {\mathbb{R}}%
			^{n}\right) }^{q_{1}}\left \Vert f\right \Vert _{\dot{K}_{p_{1}(\cdot
				)}^{\alpha ,q_{1}}(\mathbb{R}^{n})}^{q_{1}}.
	\end{eqnarray*}%
	Next, we estimate $Y$. When $D^{\gamma }A\in \dot{\Lambda}_{\beta }\left( {%
		\mathbb{R}}^{n}\right) $, from $\left( L^{p_{1}\left( \cdot \right)
	},L^{p_{2}\left( \cdot \right) }\right) $-boundedness of the $I_{\Omega
		,\phi }^{A,m}$ (see Theorem 5 in \cite{Wu0}) and (\ref{25}), we have%
	\begin{eqnarray*}
		Y &=&\sum \limits_{k=-\infty }^{\infty }2^{k\alpha q_{1}}\left( \sum
		\limits_{z=k-2}^{k+2}\Vert \left( I_{\Omega ,\phi }^{A,m}f_{z}\right) \chi
		_{k}\Vert _{L^{p_{2}(\cdot )}}\right) ^{q_{1}} \\
		&\lesssim &\sum \limits_{\left \vert \gamma \right \vert =m-1}\left \Vert
		D^{\gamma }A\right \Vert _{\dot{\Lambda}_{\beta }\left( {\mathbb{R}}%
			^{n}\right) }^{q_{1}}\sum \limits_{k=-\infty }^{\infty }2^{k\alpha
			q_{1}}\left( \sum \limits_{z=k-2}^{k+2}\Vert f_{z}\chi _{k}\Vert
		_{L^{p_{2}(\cdot )}}\right) ^{q_{1}} \\
		&\lesssim &\sum \limits_{\left \vert \gamma \right \vert =m-1}\left \Vert
		D^{\gamma }A\right \Vert _{\dot{\Lambda}_{\beta }\left( {\mathbb{R}}%
			^{n}\right) }^{q_{1}}\left[ \sum \limits_{k=-\infty }^{\infty }2^{k\alpha
			q_{1}}\Vert f\chi _{k}\Vert _{L^{p_{2}(\cdot )}}^{q_{1}}\right] \\
		&=&\sum \limits_{\left \vert \gamma \right \vert =m-1}\left \Vert D^{\gamma
		}A\right \Vert _{\dot{\Lambda}_{\beta }\left( {\mathbb{R}}^{n}\right)
		}^{q_{1}}\left \Vert f\right \Vert _{\dot{K}_{p_{1}(\cdot )}^{\alpha ,q_{1}}(%
			\mathbb{R}^{n})}^{q_{1}}.
	\end{eqnarray*}%
	Finally, we estimate $Z$. For any $k,z\in 
	\mathbb{Z}
	$ and $z\geq k+3$, from the process proving (\ref{27*}), it is easy to see
	that%
	\begin{equation*}
		\left \Vert \left( I_{\Omega ,\phi }^{A,m}f_{z}\right) \chi _{k}\right \Vert
		_{L^{p_{2}\left( \cdot \right) }}\lesssim 2^{z\left( \phi +\beta -n\right)
		}\sum \limits_{\left \vert \gamma \right \vert =m-1}\left \Vert D^{\gamma
		}A\right \Vert _{\dot{\Lambda}_{\beta }\left( {\mathbb{R}}^{n}\right) }\left
		\Vert f_{z}\right \Vert _{L^{p_{1}(\cdot )}}\left \Vert \chi _{B_{z}}\right
		\Vert _{L^{p_{^{1}}^{\prime }\left( \cdot \right) }}\left \Vert \chi
		_{B_{k}}\right \Vert _{L^{p_{2}\left( \cdot \right) }}.
	\end{equation*}%
	Note that we do not go into the details of the proof process here, as the
	proofs are similar to each other. On the other hand, using (\ref{1}) and (%
	\ref{5}), an estimate similar to (\ref{29}) yields%
	\begin{equation}
		2^{z\left( \phi +\beta -n\right) }\left \Vert \chi _{B_{z}}\right \Vert
		_{L^{p_{^{1}}^{\prime }\left( \cdot \right) }}\left \Vert \chi
		_{B_{k}}\right \Vert _{L^{p_{2}\left( \cdot \right) }}\lesssim 2^{z\left(
			\phi +\beta \right) }\left \Vert \chi _{B_{k}}\right \Vert
		_{L^{p_{^{1}}\left( \cdot \right) }}\left \Vert \chi _{B_{k}}\right \Vert
		_{L^{p_{2}\left( \cdot \right) }}^{-1}2^{n\delta _{2}\left( z-k\right) }.
		\label{40}
	\end{equation}%
	Next, we know that%
	\begin{equation}
		2^{k\left( \phi +\beta \right) }\chi _{B_{k}}\lesssim I_{\phi +\beta }\left(
		\chi _{B_{k}}\right) \left( x\right) \chi _{B_{z}}\leq I_{\phi +\beta
		}\left( \chi _{B_{k}}\right) \left( x\right) .  \label{41}
	\end{equation}%
	Moreover, similar to the estimation of (\ref{30}), we get%
	\begin{equation}
		\left \Vert \chi _{B_{k}}\right \Vert _{L^{p_{1}\left( \cdot \right)
		}}^{-1}\lesssim 2^{-nk}\left \Vert \chi _{B_{k}}\right \Vert
		_{L^{p_{1}^{\prime }\left( \cdot \right) }}\lesssim 2^{-k\left( \phi +\beta
			\right) }\left \Vert \chi _{B_{k}}\right \Vert _{L^{p_{^{2}}\left( \cdot
				\right) }}^{-1}.  \label{42}
	\end{equation}%
	Hence, combining (\ref{40})-(\ref{42}), we obtain%
	\begin{equation}
		\left \Vert \left( I_{\Omega ,\phi }^{A,m}f_{z}\right) \chi _{k}\right \Vert
		_{L^{p_{2}\left( \cdot \right) }}\lesssim 2^{\left( z-k\right) \left( \phi
			+\beta +n\delta _{2}-\alpha \right) }\sum \limits_{\left \vert \gamma
			\right \vert =m-1}\left \Vert D^{\gamma }A\right \Vert _{\dot{\Lambda}%
			_{\beta }\left( {\mathbb{R}}^{n}\right) }\left \Vert f_{z}\right \Vert
		_{L^{p_{1}(\cdot )}}.  \label{43}
	\end{equation}%
	Thus, by (\ref{43}) and the assumption $\phi +\beta +n\delta _{2}<\alpha $,
	we conclude that \ 
	\begin{eqnarray}
		Z &=&\sum \limits_{k=-\infty }^{\infty }2^{k\alpha q_{1}}\left( \sum
		\limits_{z=k+3}^{\infty }\Vert \left( I_{\Omega ,\phi }^{A,m}f_{z}\right)
		\chi _{k}\Vert _{L^{p_{2}(\cdot )}}\right) ^{q_{1}}  \notag \\
		&\lesssim &\sum \limits_{k=-\infty }^{\infty }2^{k\alpha q_{1}}\left( \sum
		\limits_{z=k+3}^{\infty }2^{\left( z-k\right) \left( \phi +\beta +n\delta
			_{2}-\alpha \right) }\sum \limits_{\left \vert \gamma \right \vert
			=m-1}\left \Vert D^{\gamma }A\right \Vert _{\dot{\Lambda}_{\beta }\left( {%
				\mathbb{R}}^{n}\right) }\left \Vert f_{z}\right \Vert _{L^{p_{1}(\cdot
				)}}\right) ^{q_{1}}  \notag \\
		&\lesssim &\sum \limits_{\left \vert \gamma \right \vert =m-1}\left \Vert
		D^{\gamma }A\right \Vert _{\dot{\Lambda}_{\beta }\left( {\mathbb{R}}%
			^{n}\right) }^{q_{1}}\sum \limits_{k=-\infty }^{\infty }\left( \sum
		\limits_{z=k+3}^{\infty }2^{\alpha z}\left \Vert f_{z}\right \Vert
		_{L^{p_{1}(\cdot )}}2^{\left( z-k\right) \left( \phi +\beta +n\delta
			_{2}-\alpha \right) }\right) ^{q_{1}}.  \notag
	\end{eqnarray}%
	To proceed, we consider $Z$ in two cases. Indeed, if $0<q_{1}\leq 1$, then
	by replacing $q_{1}$ with $\frac{q_{1}}{q_{2}}$ in (\ref{25}) to get%
	\begin{eqnarray*}
		Z &\lesssim &\sum \limits_{\left \vert \gamma \right \vert =m-1}\left \Vert
		D^{\gamma }A\right \Vert _{\dot{\Lambda}_{\beta }\left( {\mathbb{R}}%
			^{n}\right) }^{q_{1}}\sum \limits_{k=-\infty }^{\infty }\sum
		\limits_{z=k+3}^{\infty }2^{\alpha zq_{1}}\left \Vert f_{z}\right \Vert
		_{L^{p_{1}(\cdot )}}^{q_{1}}2^{\left( z-k\right) \left( \phi +\beta +n\delta
			_{2}-\alpha \right) q_{1}} \\
		&\lesssim &\sum \limits_{\left \vert \gamma \right \vert =m-1}\left \Vert
		D^{\gamma }A\right \Vert _{\dot{\Lambda}_{\beta }\left( {\mathbb{R}}%
			^{n}\right) }^{q_{1}}\left \Vert f\right \Vert _{\dot{K}_{p_{1}(\cdot
				)}^{\alpha ,q_{1}}(\mathbb{R}^{n})}^{q_{1}}.
	\end{eqnarray*}%
	Now, let $1<q_{1}<\infty $ and $\frac{1}{q_{1}}+\frac{1}{q_{1}^{\prime }}=1$%
	. By (\ref{3}), we have%
	\begin{eqnarray*}
		Z &\lesssim &\sum \limits_{\left \vert \gamma \right \vert =m-1}\left \Vert
		D^{\gamma }A\right \Vert _{\dot{\Lambda}_{\beta }\left( {\mathbb{R}}%
			^{n}\right) }^{q_{1}}\sum \limits_{k=-\infty }^{\infty }\sum
		\limits_{z=k+3}^{\infty }2^{\alpha zq_{1}}\left \Vert f_{z}\right \Vert
		_{L^{p_{1}(\cdot )}}^{q_{1}}2^{\left( z-k\right) \left( \phi +\beta +n\delta
			_{2}-\alpha \right) \frac{q_{1}}{2}} \\
		&&\times \left( \sum \limits_{z=k+3}^{\infty }2^{\left( z-k\right) \left(
			\phi +\beta +n\delta _{2}-\alpha \right) \frac{q_{1}^{\prime }}{2}}\right) ^{%
			\frac{q_{1}}{q_{1}^{\prime }}} \\
		&\lesssim &\sum \limits_{\left \vert \gamma \right \vert =m-1}\left \Vert
		D^{\gamma }A\right \Vert _{\dot{\Lambda}_{\beta }\left( {\mathbb{R}}%
			^{n}\right) }^{q_{1}}\sum \limits_{z=-\infty }^{\infty }2^{\alpha
			zq_{1}}\left \Vert f_{z}\right \Vert _{L^{p_{1}(\cdot )}}^{q_{1}}\sum
		\limits_{k=z-3}^{\infty }2^{\left( z-k\right) \left( \phi +\beta +n\delta
			_{2}-\alpha \right) \frac{q_{1}}{2}} \\
		&\lesssim &\sum \limits_{\left \vert \gamma \right \vert =m-1}\left \Vert
		D^{\gamma }A\right \Vert _{\dot{\Lambda}_{\beta }\left( {\mathbb{R}}%
			^{n}\right) }^{q_{1}}\left \Vert f\right \Vert _{\dot{K}_{p_{1}(\cdot
				)}^{\alpha ,q_{1}}(\mathbb{R}^{n})}^{q_{1}}.
	\end{eqnarray*}%
	Thus, by introducing approximations of $X$, $Y$ and $Z$ into (\ref{101}), (%
	\ref{5*}) is attained.
	
	We are now at the point of proving (\ref{6}). If we remember 
	\begin{equation*}
		\widetilde{T}_{\left \vert \Omega \right \vert ,\phi }^{A}\left( \left \vert
		f\right \vert \right) (x)=\int \limits_{{\mathbb{R}^{n}}}\frac{\left \vert
			\Omega (x-y)\right \vert }{|x-y|^{n-\phi +m-1}}\left \vert R_{m}\left(
		A;x,y\right) \right \vert \left \vert f(y)\right \vert dy\qquad 0<\phi <n,
	\end{equation*}%
	it is easy to see that the conclusions of (\ref{5*}) also hold for $%
	\widetilde{T}_{\left \vert \Omega \right \vert ,\phi }^{A}$. Therefore, (\ref%
	{6}) is a direct consequence of Lemma 3.2 in \cite{Gurbuz} and the above
	conclusions. This completes the proof of Theorem \ref{Theorem}.
\end{proof}

Now, we prove Theorem \ref{Theorem1}.

\textbf{Proof of Theorem \ref{Theorem1}}

\begin{proof}
	Assume $f\in M\dot{K}_{p_{1}(\cdot )}^{\alpha ,q_{1}}(\mathbb{R}^{n})$, $%
	f_{z}:=f\cdot \chi _{z}$ and $f=\sum \limits_{z=-\infty }^{\infty }f_{z}$ $%
	\left( z\in 
	\mathbb{Z}
	\right) $. Then, using (\ref{25}), we get%
	\begin{eqnarray*}
		\left \Vert I_{\Omega ,\phi }^{A,m}f\right \Vert _{M\dot{K}_{p_{2}(\cdot
				)}^{\alpha ,q_{2}}(\mathbb{R}^{n})}^{q_{1}} &=&\sup_{L\in \mathbb{%
				\mathbb{Z}
		}}2^{-L\lambda q_{1}}\left( \sum \limits_{k=-\infty }^{L}2^{k\alpha
			q_{2}}\Vert \left( I_{\Omega ,\phi }^{A,m}f\right) \chi _{k}\Vert
		_{L^{p_{2}(\cdot )}}^{q_{2}}\right) ^{\frac{q_{1}}{q_{2}}} \\
		&\lesssim &\sup_{L\in \mathbb{%
				\mathbb{Z}
		}}2^{-L\lambda q_{1}}\sum \limits_{k=-\infty }^{L}2^{k\alpha q_{1}}\Vert
		\left( I_{\Omega ,\phi }^{A,m}f\right) \chi _{k}\Vert _{L^{p_{2}(\cdot
				)}}^{q_{1}} \\
		&\lesssim &\sup_{L\in \mathbb{%
				\mathbb{Z}
		}}2^{-L\lambda q_{1}}\left[ \sum \limits_{k=-\infty }^{L}2^{k\alpha
			q_{1}}\left( \sum \limits_{z=-\infty }^{k-3}\Vert \left( I_{\Omega ,\phi
		}^{A,m}f_{z}\right) \chi _{k}\Vert _{L^{p_{2}(\cdot )}}\right) ^{q_{1}}%
		\right] \\
		&&+\sup_{L\in \mathbb{%
				\mathbb{Z}
		}}2^{-L\lambda q_{1}}\left[ \sum \limits_{k=-\infty }^{L}2^{k\alpha
			q_{1}}\left( \sum \limits_{z=k-2}^{k+2}\Vert \left( I_{\Omega ,\phi
		}^{A,m}f_{z}\right) \chi _{k}\Vert _{L^{p_{2}(\cdot )}}\right) ^{q_{1}}%
		\right]
	\end{eqnarray*}%
	\begin{eqnarray}
		&&+\sup_{L\in \mathbb{%
				\mathbb{Z}
		}}2^{-L\lambda q_{1}}\left[ \sum \limits_{k=-\infty }^{L}2^{k\alpha
			q_{1}}\left( \sum \limits_{z=k+3}^{\infty }\Vert \left( I_{\Omega ,\phi
		}^{A,m}f_{z}\right) \chi _{k}\Vert _{L^{p_{2}(\cdot )}}\right) ^{q_{1}}%
		\right]  \notag \\
		&=&:X_{1}+Y_{1}+Z_{1}.  \label{102*}
	\end{eqnarray}%
	First, we estimate $X_{1}$. Similar to the estimation method of $X$ in
	Theorem \ref{Theorem}, by virtue of (\ref{37})%
	\begin{eqnarray}
		X_{1} &=&\sup_{L\in \mathbb{%
				\mathbb{Z}
		}}2^{-L\lambda q_{1}}\left[ \sum \limits_{k=-\infty }^{L}2^{k\alpha
			q_{1}}\left( \sum \limits_{z=-\infty }^{k-3}\Vert \left( I_{\Omega ,\phi
		}^{A,m}f_{z}\right) \chi _{k}\Vert _{L^{p_{2}(\cdot )}}\right) ^{q_{1}}%
		\right]  \notag \\
		&\lesssim &\sup_{L\in \mathbb{%
				\mathbb{Z}
		}}2^{-L\lambda q_{1}}\left[ \sum \limits_{k=-\infty }^{L}2^{k\alpha
			q_{1}}\left( \sum \limits_{z=-\infty }^{k-3}2^{\left( k-z\right) \left( \phi
			+\beta -n\delta _{1}+\frac{n-1}{s}\right) }\sum \limits_{\left \vert \gamma
			\right \vert =m-1}\left \Vert D^{\gamma }A\right \Vert _{\dot{\Lambda}%
			_{\beta }\left( {\mathbb{R}}^{n}\right) }\left \Vert f_{z}\right \Vert
		_{L^{p_{1}(\cdot )}}\right) ^{q_{1}}\right]  \notag \\
		&\lesssim &\sum \limits_{\left \vert \gamma \right \vert =m-1}\left \Vert
		D^{\gamma }A\right \Vert _{\dot{\Lambda}_{\beta }\left( {\mathbb{R}}%
			^{n}\right) }^{q_{1}}\sup_{L\in \mathbb{%
				\mathbb{Z}
		}}2^{-L\lambda q_{1}}\left[ \sum \limits_{k=-\infty }^{L}\left( \sum
		\limits_{z=-\infty }^{k-3}2^{\alpha z}\left \Vert f_{z}\right \Vert
		_{L^{p_{1}(\cdot )}}2^{\left( k-z\right) \left( \phi +\beta -n\delta
			_{1}+\alpha +\frac{n-1}{s}\right) }\right) ^{q_{1}}\right] .  \notag
	\end{eqnarray}%
	On the other hand, we know that following fact:%
	\begin{eqnarray}
		\left \Vert f_{z}\right \Vert _{L^{p_{1}(\cdot )}} &=&2^{-\alpha z}\left(
		2^{\alpha zq_{1}}\left \Vert f\chi _{z}\right \Vert _{L^{p_{1}(\cdot
				)}}^{q_{1}}\right) ^{\frac{1}{q_{1}}}  \notag \\
		&\leq &2^{-\alpha z}\left( \sum \limits_{i=-\infty }^{z}2^{\alpha
			iq_{1}}\left \Vert f\chi _{i}\right \Vert _{L^{p_{1}(\cdot
				)}}^{q_{1}}\right) ^{\frac{1}{q_{1}}}  \notag \\
		&=&2^{z\left( \lambda -\alpha \right) }\left[ 2^{-\lambda z}\left( \sum
		\limits_{i=-\infty }^{z}2^{\alpha iq_{1}}\left \Vert f\chi _{i}\right \Vert
		_{L^{p_{1}(\cdot )}}^{q_{1}}\right) ^{\frac{1}{q_{1}}}\right]  \notag \\
		&\lesssim &2^{z\left( \lambda -\alpha \right) }\left \Vert f\right \Vert _{M%
			\dot{K}_{p_{1}(\cdot )}^{\alpha ,q_{1}}(\mathbb{R}^{n})}^{q_{1}}.  \label{38}
	\end{eqnarray}%
	To continue calculating $X_{1}$, we consider the two cases $0<q_{1}\leq 1$
	and $1<q_{1}<\infty $, respectively.
	
	If $0<q_{1}\leq 1$ and $\alpha <$ $n\delta _{1}+\lambda -\left( \phi +\beta +%
	\frac{n-1}{s}\right) $, then using (\ref{25}) and (\ref{38}), we get%
	\begin{eqnarray*}
		X_{1} &\lesssim &\sum \limits_{\left \vert \gamma \right \vert =m-1}\left
		\Vert D^{\gamma }A\right \Vert _{\dot{\Lambda}_{\beta }\left( {\mathbb{R}}%
			^{n}\right) }^{q_{1}}\left \Vert f\right \Vert _{M\dot{K}_{p_{1}(\cdot
				)}^{\alpha ,q_{1}}(\mathbb{R}^{n})}^{q_{1}} \\
		&&\times \sup_{L\in \mathbb{%
				\mathbb{Z}
		}}2^{-L\lambda q_{1}}\left[ \sum \limits_{k=-\infty }^{L}2^{k\lambda
			q_{1}}\sum \limits_{z=-\infty }^{k-3}2^{\left( k-z\right) \left( \phi +\beta
			-n\delta _{1}-\lambda +\alpha +\frac{n-1}{s}\right) q_{1}}\right] \\
		&\lesssim &\sum \limits_{\left \vert \gamma \right \vert =m-1}\left \Vert
		D^{\gamma }A\right \Vert _{\dot{\Lambda}_{\beta }\left( {\mathbb{R}}%
			^{n}\right) }^{q_{1}}\left \Vert f\right \Vert _{M\dot{K}_{p_{1}(\cdot
				)}^{\alpha ,q_{1}}(\mathbb{R}^{n})}^{q_{1}}\sup_{L\in \mathbb{%
				\mathbb{Z}
		}}2^{-L\lambda q_{1}}\left( \sum \limits_{k=-\infty }^{L}2^{k\lambda
			q_{1}}\right) \\
		&\lesssim &\sum \limits_{\left \vert \gamma \right \vert =m-1}\left \Vert
		D^{\gamma }A\right \Vert _{\dot{\Lambda}_{\beta }\left( {\mathbb{R}}%
			^{n}\right) }^{q_{1}}\left \Vert f\right \Vert _{M\dot{K}_{p_{1}(\cdot
				)}^{\alpha ,q_{1}}(\mathbb{R}^{n})}^{q_{1}}.
	\end{eqnarray*}%
	If $1<q_{1}<\infty $ and $\alpha <$ $n\delta _{1}+\lambda -\left( \phi
	+\beta +\frac{n-1}{s}\right) $, then we use (\ref{3}), (\ref{38}) and obtain%
	\begin{equation*}
		X_{1}\lesssim \sum \limits_{\left \vert \gamma \right \vert =m-1}\left
		\Vert D^{\gamma }A\right \Vert _{\dot{\Lambda}_{\beta }\left( {\mathbb{R}}%
			^{n}\right) }^{q_{1}}\left \Vert f\right \Vert _{M\dot{K}_{p_{1}(\cdot
				)}^{\alpha ,q_{1}}(\mathbb{R}^{n})}^{q_{1}}
	\end{equation*}%
	\begin{eqnarray*}
		&&\times \sup_{L\in \mathbb{%
				\mathbb{Z}
		}}2^{-L\lambda q_{1}}\left[ 
		\begin{array}{c}
			\sum \limits_{k=-\infty }^{L}2^{k\lambda q_{1}}\sum \limits_{z=-\infty
			}^{k-3}2^{\left( k-z\right) \left( \phi +\beta -n\delta _{1}-\lambda +\alpha
				+\frac{n-1}{s}\right) \frac{q_{1}}{2}} \\ 
			\times \left( \sum \limits_{z=-\infty }^{k-3}2^{\left( k-z\right) \left(
				\phi +\beta -n\delta _{1}-\lambda +\alpha +\frac{n-1}{s}\right) \frac{%
					q_{1}^{\prime }}{2}}\right) ^{\frac{q_{1}}{q_{1}^{\prime }}}%
		\end{array}%
		\right] \\
		&\lesssim &\sum \limits_{\left \vert \gamma \right \vert =m-1}\left \Vert
		D^{\gamma }A\right \Vert _{\dot{\Lambda}_{\beta }\left( {\mathbb{R}}%
			^{n}\right) }^{q_{1}}\left \Vert f\right \Vert _{M\dot{K}_{p_{1}(\cdot
				)}^{\alpha ,q_{1}}(\mathbb{R}^{n})}^{q_{1}}\sup_{L\in \mathbb{%
				\mathbb{Z}
		}}2^{-L\lambda q_{1}}\left( \sum \limits_{k=-\infty }^{L}2^{k\lambda
			q_{1}}\right) \\
		&\lesssim &\sum \limits_{\left \vert \gamma \right \vert =m-1}\left \Vert
		D^{\gamma }A\right \Vert _{\dot{\Lambda}_{\beta }\left( {\mathbb{R}}%
			^{n}\right) }^{q_{1}}\left \Vert f\right \Vert _{M\dot{K}_{p_{1}(\cdot
				)}^{\alpha ,q_{1}}(\mathbb{R}^{n})}^{q_{1}}.
	\end{eqnarray*}
	
	Next, we estimate $Y_{1}$. If $D^{\gamma }A\in \dot{\Lambda}_{\beta }\left( {%
		\mathbb{R}}^{n}\right) $, then from the fact that $I_{\Omega ,\phi }^{A,m}$
	is bounded from $L^{p_{1}\left( \cdot \right) }$ to $L^{p_{2}\left( \cdot
		\right) }$ (see Theorem 5 in \cite{Wu0}) and (\ref{25}), it follows that%
	\begin{eqnarray*}
		Y_{1} &=&\sup_{L\in \mathbb{%
				\mathbb{Z}
		}}2^{-L\lambda q_{1}}\left[ \sum \limits_{k=-\infty }^{L}2^{k\alpha
			q_{1}}\left( \sum \limits_{z=k-2}^{k+2}\Vert \left( I_{\Omega ,\phi
		}^{A,m}f_{z}\right) \chi _{k}\Vert _{L^{p_{2}(\cdot )}}\right) ^{q_{1}}%
		\right] \\
		&\lesssim &\sum \limits_{\left \vert \gamma \right \vert =m-1}\left \Vert
		D^{\gamma }A\right \Vert _{\dot{\Lambda}_{\beta }\left( {\mathbb{R}}%
			^{n}\right) }^{q_{1}}\sup_{L\in \mathbb{%
				\mathbb{Z}
		}}2^{-L\lambda q_{1}}\left[ \sum \limits_{k=-\infty }^{L}2^{k\alpha
			q_{1}}\left( \sum \limits_{z=k-2}^{k+2}\Vert f_{z}\chi _{k}\Vert
		_{L^{p_{2}(\cdot )}}\right) ^{q_{1}}\right] \text{\  \  \  \ } \\
		&\lesssim &\sum \limits_{\left \vert \gamma \right \vert =m-1}\left \Vert
		D^{\gamma }A\right \Vert _{\dot{\Lambda}_{\beta }\left( {\mathbb{R}}%
			^{n}\right) }^{q_{1}}\sup_{L\in \mathbb{%
				\mathbb{Z}
		}}2^{-L\lambda q_{1}}\left[ \sum \limits_{k=-\infty }^{L}2^{k\alpha
			q_{1}}\Vert f\chi _{k}\Vert _{L^{p_{2}(\cdot )}}^{q_{1}}\right] \\
		&=&\sum \limits_{\left \vert \gamma \right \vert =m-1}\left \Vert D^{\gamma
		}A\right \Vert _{\dot{\Lambda}_{\beta }\left( {\mathbb{R}}^{n}\right)
		}^{q_{1}}\left \Vert f\right \Vert _{M\dot{K}_{p_{1}(\cdot )}^{\alpha
				,q_{1}}(\mathbb{R}^{n})}^{q_{1}}.
	\end{eqnarray*}%
	Now, we estimate $Z_{1}$. By (\ref{43}) and the assumption $\phi +\beta
	+n\delta _{2}<\alpha $, we know that \ 
	\begin{eqnarray*}
		Z_{1} &=&\sup_{L\in \mathbb{%
				\mathbb{Z}
		}}2^{-L\lambda q_{1}}\left[ \sum \limits_{k=-\infty }^{L}2^{k\alpha
			q_{1}}\left( \sum \limits_{z=k+3}^{\infty }\Vert \left( I_{\Omega ,\phi
		}^{A,m}f_{z}\right) \chi _{k}\Vert _{L^{p_{2}(\cdot )}}\right) ^{q_{1}}%
		\right] \\
		&\lesssim &\sum \limits_{\left \vert \gamma \right \vert =m-1}\left \Vert
		D^{\gamma }A\right \Vert _{\dot{\Lambda}_{\beta }\left( {\mathbb{R}}%
			^{n}\right) }^{q_{1}}\sup_{L\in \mathbb{%
				\mathbb{Z}
		}}2^{-L\lambda q_{1}}\left[ \sum \limits_{k=-\infty }^{L}\left( \sum
		\limits_{z=k+3}^{\infty }2^{\alpha z}\left \Vert f_{z}\right \Vert
		_{L^{p_{1}(\cdot )}}2^{\left( z-k\right) \left( \phi +\beta +n\delta
			_{2}-\alpha \right) }\right) ^{q_{1}}\right] \\
		&\lesssim &\sum \limits_{\left \vert \gamma \right \vert =m-1}\left \Vert
		D^{\gamma }A\right \Vert _{\dot{\Lambda}_{\beta }\left( {\mathbb{R}}%
			^{n}\right) }^{q_{1}}\sup_{L\in \mathbb{%
				\mathbb{Z}
		}}2^{-L\lambda q_{1}}\left[ \sum \limits_{k=-\infty }^{L}\left( \sum
		\limits_{z=k+3}^{L}2^{\alpha z}\left \Vert f_{z}\right \Vert
		_{L^{p_{1}(\cdot )}}2^{\left( z-k\right) \left( \phi +\beta +n\delta
			_{2}-\alpha \right) }\right) ^{q_{1}}\right] \\
		&&+\sum \limits_{\left \vert \gamma \right \vert =m-1}\left \Vert D^{\gamma
		}A\right \Vert _{\dot{\Lambda}_{\beta }\left( {\mathbb{R}}^{n}\right)
		}^{q_{1}}\sup_{L\in \mathbb{%
				\mathbb{Z}
		}}2^{-L\lambda q_{1}}\left[ \sum \limits_{k=-\infty }^{L}\left( \sum
		\limits_{z=L+1}^{\infty }2^{\alpha z}\left \Vert f_{z}\right \Vert
		_{L^{p_{1}(\cdot )}}2^{\left( z-k\right) \left( \phi +\beta +n\delta
			_{2}-\alpha \right) }\right) ^{q_{1}}\right] \\
		&=&:Z_{11}+Z_{12}.
	\end{eqnarray*}%
	To continue estimating $Z_{1}$, we look at two scenarios: $0<q_{1}\leq 1$
	and $1<q_{1}<\infty $.
	
	When $0<q_{1}\leq 1$ and $\phi +\beta +n\delta _{2}+\lambda <\alpha $, then
	using (\ref{25}) and (\ref{38}), we get%
	\begin{eqnarray*}
		Z_{1} &\lesssim &\sum \limits_{\left \vert \gamma \right \vert =m-1}\left
		\Vert D^{\gamma }A\right \Vert _{\dot{\Lambda}_{\beta }\left( {\mathbb{R}}%
			^{n}\right) }^{q_{1}}\left \Vert f\right \Vert _{M\dot{K}_{p_{1}(\cdot
				)}^{\alpha ,q_{1}}(\mathbb{R}^{n})}^{q_{1}}\sup_{L\in \mathbb{%
				\mathbb{Z}
		}}2^{-L\lambda q_{1}}\left[ \sum \limits_{k=-\infty }^{L}2^{k\lambda
			q_{1}}\sum \limits_{z=k+3}^{L}2^{\left( z-k\right) \left( \phi +\beta
			+n\delta _{2}+\lambda -\alpha \right) q_{1}}\right] \\
		&&+\sum \limits_{\left \vert \gamma \right \vert =m-1}\left \Vert D^{\gamma
		}A\right \Vert _{\dot{\Lambda}_{\beta }\left( {\mathbb{R}}^{n}\right)
		}^{q_{1}}\sup_{L\in \mathbb{%
				\mathbb{Z}
		}}2^{-L\lambda q_{1}}\left[ \sum \limits_{k=-\infty }^{L}\sum
		\limits_{z=L+1}^{\infty }2^{\alpha zq_{1}}\left \Vert f_{z}\right \Vert
		_{L^{p_{1}(\cdot )}}^{q_{1}}2^{\left( z-k\right) \left( \phi +\beta +n\delta
			_{2}+\lambda -\alpha \right) q_{1}}\right] \\
		&\lesssim &\sum \limits_{\left \vert \gamma \right \vert =m-1}\left \Vert
		D^{\gamma }A\right \Vert _{\dot{\Lambda}_{\beta }\left( {\mathbb{R}}%
			^{n}\right) }^{q_{1}}\left \Vert f\right \Vert _{M\dot{K}_{p_{1}(\cdot
				)}^{\alpha ,q_{1}}(\mathbb{R}^{n})}^{q_{1}}\sup_{L\in \mathbb{%
				\mathbb{Z}
		}}2^{-L\lambda q_{1}}\left( \sum \limits_{k=-\infty }^{L}2^{k\lambda
			q_{1}}\right) \\
		&&+\sum \limits_{\left \vert \gamma \right \vert =m-1}\left \Vert D^{\gamma
		}A\right \Vert _{\dot{\Lambda}_{\beta }\left( {\mathbb{R}}^{n}\right)
		}^{q_{1}}\sup_{L\in \mathbb{%
				\mathbb{Z}
		}}2^{-L\lambda q_{1}} \\
		&&\times \left[ \sum \limits_{k=-\infty }^{L}\sum \limits_{z=L+1}^{\infty
		}2^{z\lambda q_{1}}2^{\left( z-k\right) \left( \phi +\beta +n\delta
			_{2}+\lambda -\alpha \right) q_{1}}2^{-z\lambda q_{1}}\sum
		\limits_{l=-\infty }^{z}2^{\alpha lq_{1}}\left \Vert f_{l}\right \Vert
		_{L^{p_{1}(\cdot )}}^{q_{1}}\right] \\
		&\lesssim &\sum \limits_{\left \vert \gamma \right \vert =m-1}\left \Vert
		D^{\gamma }A\right \Vert _{\dot{\Lambda}_{\beta }\left( {\mathbb{R}}%
			^{n}\right) }^{q_{1}}\left \Vert f\right \Vert _{M\dot{K}_{p_{1}(\cdot
				)}^{\alpha ,q_{1}}(\mathbb{R}^{n})}^{q_{1}}.
	\end{eqnarray*}%
	If $1<q_{1}<\infty $ and $\phi +\beta +n\delta _{2}+\lambda <\alpha $, then
	we use (\ref{3}), (\ref{38}) and obtain%
	\begin{eqnarray*}
		Z_{11} &\lesssim &\sum \limits_{\left \vert \gamma \right \vert =m-1}\left
		\Vert D^{\gamma }A\right \Vert _{\dot{\Lambda}_{\beta }\left( {\mathbb{R}}%
			^{n}\right) }^{q_{1}}\left \Vert f\right \Vert _{M\dot{K}_{p_{1}(\cdot
				)}^{\alpha ,q_{1}}(\mathbb{R}^{n})}^{q_{1}} \\
		&&\times \sup_{L\in \mathbb{%
				\mathbb{Z}
		}}2^{-L\lambda q_{1}}\left[ 
		\begin{array}{c}
			\sum \limits_{k=-\infty }^{L}2^{k\lambda q_{1}}\sum
			\limits_{z=k+3}^{L}2^{\left( z-k\right) \left( \phi +\beta +n\delta
				_{2}+\lambda -\alpha \right) \frac{q_{1}}{2}} \\ 
			\times \left( \sum \limits_{z=-\infty }^{k+3}2^{\left( z-k\right) \left(
				\phi +\beta +n\delta _{2}+\lambda -\alpha \right) \frac{q_{1}^{\prime }}{2}%
			}\right) ^{\frac{q_{1}}{q_{1}^{\prime }}}%
		\end{array}%
		\right] \\
		&\lesssim &\sum \limits_{\left \vert \gamma \right \vert =m-1}\left \Vert
		D^{\gamma }A\right \Vert _{\dot{\Lambda}_{\beta }\left( {\mathbb{R}}%
			^{n}\right) }^{q_{1}}\left \Vert f\right \Vert _{M\dot{K}_{p_{1}(\cdot
				)}^{\alpha ,q_{1}}(\mathbb{R}^{n})}^{q_{1}}\sup_{L\in \mathbb{%
				\mathbb{Z}
		}}2^{-L\lambda q_{1}}\left( \sum \limits_{k=-\infty }^{L}2^{k\lambda
			q_{1}}\right) \\
		&\lesssim &\sum \limits_{\left \vert \gamma \right \vert =m-1}\left \Vert
		D^{\gamma }A\right \Vert _{\dot{\Lambda}_{\beta }\left( {\mathbb{R}}%
			^{n}\right) }^{q_{1}}\left \Vert f\right \Vert _{M\dot{K}_{p_{1}(\cdot
				)}^{\alpha ,q_{1}}(\mathbb{R}^{n})}^{q_{1}}.
	\end{eqnarray*}%
	Also, when $1<q_{1}<\infty $ and $\phi +\beta +n\delta _{2}+\lambda <\alpha $%
	, then we use (\ref{3}), (\ref{38}) and get%
	\begin{eqnarray*}
		Z_{12} &\lesssim &\sum \limits_{\left \vert \gamma \right \vert =m-1}\left
		\Vert D^{\gamma }A\right \Vert _{\dot{\Lambda}_{\beta }\left( {\mathbb{R}}%
			^{n}\right) }^{q_{1}}\sup_{L\in \mathbb{%
				\mathbb{Z}
		}}2^{-L\lambda q_{1}} \\
		&&\times \left[ \sum \limits_{k=-\infty }^{L}\left( \sum
		\limits_{z=L+1}^{\infty }2^{\alpha z}\left \Vert f_{z}\right \Vert
		_{L^{p_{1}(\cdot )}}2^{\left( z-k\right) \frac{\left( \phi +\beta +n\delta
				_{2}+\lambda -\alpha \right) }{2}}2^{\left( z-k\right) \frac{\left( \phi
				+\beta +n\delta _{2}-\lambda -\alpha \right) }{2}}\right) ^{q_{1}}\right] \\
		&\lesssim &\sum \limits_{\left \vert \gamma \right \vert =m-1}\left \Vert
		D^{\gamma }A\right \Vert _{\dot{\Lambda}_{\beta }\left( {\mathbb{R}}%
			^{n}\right) }^{q_{1}}\sup_{L\in \mathbb{%
				\mathbb{Z}
		}}2^{-L\lambda q_{1}} \\
		&&\times \left[ \sum \limits_{k=-\infty }^{L}\sum \limits_{z=L+1}^{\infty
		}2^{z\lambda q_{1}}2^{\left( z-k\right) \frac{\left( \phi +\beta +n\delta
				_{2}+\lambda -\alpha \right) }{2}q_{1}}2^{-z\lambda q_{1}}\sum
		\limits_{l=-\infty }^{z}2^{\alpha lq_{1}}\left \Vert f_{l}\right \Vert
		_{L^{p_{1}(\cdot )}}^{q_{1}}\right] \\
		&\lesssim &\sum \limits_{\left \vert \gamma \right \vert =m-1}\left \Vert
		D^{\gamma }A\right \Vert _{\dot{\Lambda}_{\beta }\left( {\mathbb{R}}%
			^{n}\right) }^{q_{1}}\left \Vert f\right \Vert _{M\dot{K}_{p_{1}(\cdot
				)}^{\alpha ,q_{1}}(\mathbb{R}^{n})}^{q_{1}}\sup_{L\in \mathbb{%
				\mathbb{Z}
		}}2^{-L\lambda q_{1}}
	\end{eqnarray*}%
	\begin{eqnarray*}
		&&\times \left[ \sum \limits_{k=-\infty }^{L}2^{k\lambda q_{1}}\sum
		\limits_{z=L+1}^{\infty }2^{\left( z-k\right) \frac{\left( \phi +\beta
				+n\delta _{2}+\lambda -\alpha \right) }{2}q_{1}}\right] \\
		&\lesssim &\sum \limits_{\left \vert \gamma \right \vert =m-1}\left \Vert
		D^{\gamma }A\right \Vert _{\dot{\Lambda}_{\beta }\left( {\mathbb{R}}%
			^{n}\right) }^{q_{1}}\left \Vert f\right \Vert _{M\dot{K}_{p_{1}(\cdot
				)}^{\alpha ,q_{1}}(\mathbb{R}^{n})}^{q_{1}}.
	\end{eqnarray*}%
	Thus, combining the estimates of $Z_{11}$ and $Z_{12}$, we find that 
	\begin{equation*}
		Z_{1}\lesssim \sum \limits_{\left \vert \gamma \right \vert =m-1}\left
		\Vert D^{\gamma }A\right \Vert _{\dot{\Lambda}_{\beta }\left( {\mathbb{R}}%
			^{n}\right) }^{q_{1}}\left \Vert f\right \Vert _{M\dot{K}_{p_{1}(\cdot
				)}^{\alpha ,q_{1}}(\mathbb{R}^{n})}^{q_{1}}.
	\end{equation*}%
	Thus, by introducing estimates of $X_{1}$, $Y_{1}$ and $Z_{1}$ into (\ref%
	{102*}), (\ref{7}) is obtained.
	
	Finally, it is simple to demonstrate that (\ref{8}). Indeed, we first know
	that%
	\begin{equation}
		\widetilde{T}_{\left \vert \Omega \right \vert ,\phi }^{A}\left( \left \vert
		f\right \vert \right) (x)\geq M_{\Omega ,\phi }^{A}f(x)  \label{103}
	\end{equation}%
	for $x\in {\mathbb{R}^{n}}$ and $0<\phi <n$ (see Lemma 3.2 in \cite{Gurbuz}%
	). Next, from the process proving (\ref{7}), the conclusions of (\ref{7})
	also hold for $\widetilde{T}_{\left \vert \Omega \right \vert ,\phi }^{A}$.
	Thus, combining this with (\ref{103}), we can immediately obtain (\ref{8}),
	which concludes the proof of Theorem \ref{Theorem1}.
\end{proof}
\begin{acknowledgement}
The author would like to express his deep gratitude to the Dr. Liwei Wang
(Anhui Polytechnic University, Wuhu, China ) for carefully reading the
manuscript and giving some valuable suggestions and important comments
during the process of this study.
\end{acknowledgement}

\end{document}